\documentclass[12pt]{amsart}
\usepackage{amsmath,times,epsfig,amssymb,amsbsy,amscd,amsfonts,amstext,color,bm}
\usepackage[margin=1in]{geometry}
\usepackage{enumerate}
\usepackage{placeins}
\usepackage{array}
\usepackage{tabulary}
\usepackage{multirow}
\usepackage{graphicx}
\usepackage{hyperref}
\usepackage{hhline}
\usepackage{soul}
\usepackage{multirow}
\usepackage[table, svgnames, dvipsnames]{xcolor}
\usepackage{makecell, cellspace, caption}
\usepackage{subcaption}
\usepackage[color,matrix,arrow]{xypic}
\usepackage{tikz}
        \usetikzlibrary{arrows.meta}
        \usetikzlibrary{arrows}
        \usetikzlibrary{quotes}
        \usetikzlibrary{graphs}
        \usetikzlibrary{positioning}
\usepackage{tikz-cd}
\hypersetup{
  colorlinks   = true, 
  urlcolor     = blue, 
  linkcolor    = blue, 
  citecolor   = red 
}

\theoremstyle{plain}
\newtheorem{theorem}{Theorem}[section]
\newtheorem{corollary}[theorem]{Corollary}
\newtheorem{lemma}[theorem]{Lemma}

\newtheorem{proposition}[theorem]{Proposition}

\makeatletter
    
    \@addtoreset{equation}{section}
  \makeatother

\theoremstyle{definition}

\newtheorem{example}[theorem]{Example}

\newtheorem{problem}[theorem]{Problem}

\theoremstyle{remark}
\newtheorem{remark}[theorem]{Remark}

\newcommand{\A}{\mathcal{A}}

\newcommand{\scC}{\mathcal{C}}

\newcommand{\Hom}{\operatorname{Hom}}

\newcolumntype{K}[1]{>{\centering\arraybackslash}p{#1}}

\newcommand{\ehr}{\operatorname{ehr}} 
\newcommand{\Hilb}{\operatorname{Hilb}} 
\newcommand{\TL}{\mathrm{TL}}
\newcommand{\intvec}{\bm n}

\begin{document}

\title{Ehrhart quasi-polynomials and parallel translations}

\author{Akihiro Higashitani}
\address{Department of Pure and Applied Mathematics, Graduate School of Information Science and Technology, Osaka University, Suita, Osaka 565-0871, Japan}
\email{higashitani@ist.osaka-u.ac.jp}
\author{Satoshi Murai}
\address{Department of Mathematics, Faculty of Education Waseda University, 1-6-1 NishiWaseda, Shinjuku, Tokyo 169-8050, Japan}
\email{s-murai@waseda.jp}
\author{Masahiko Yoshinaga}
\address{Masahiko Yoshinaga, Department of Mathematics, Graduate School of Science, Osaka University, Toyonaka 560-0043, Japan.}


\maketitle

\begin{abstract}
Given a rational polytope $P \subset \mathbb R^d$,
the numerical function counting lattice points in the integral dilations of $P$ is known to become a quasi-polynomial, called the Ehrhart quasi-polynomial $\ehr_P$ of $P$.
In this paper we study the following problem:
Given a rational $d$-polytope $P \subset \mathbb R^d$, is there a nice way to know Ehrhart quasi-polynomials of translated polytopes $P+ \bm v$ for all $\bm v \in \mathbb Q^d$?
We provide a way to compute such Ehrhart quasi-polynomials using a certain toric arrangement and lattice point counting functions of translated cones of $P$.
This method allows us to visualize how constituent polynomials of $\ehr_{P+\bm v}$ change in the torus $\mathbb R^d/\mathbb Z^d$.
We also prove that information of $\ehr_{P+\bm v}$ for all $\bm v \in \mathbb Q^d$ determines the rational $d$-polytope $P \subset \mathbb R^d$ up to translations by integer vectors,
and characterize all rational $d$-polytopes $P \subset \mathbb R^d$ such that $\ehr_{P+\bm v}$ is symmetric for all $\bm v \in \mathbb Q^d$.
\end{abstract}

\section{Introduction}
Enumerations of lattice points in a convex polytope is a classical important theme relating to algebra, combinatorics and geometry of convex polytopes.
A fundamental result on this subject is Ehrhart's result which says that,
for any rational polytope $P \subset \mathbb R^d$,
the function $\mathbb Z_{\geq  0} \ni t \mapsto \# (tP \cap \mathbb Z^d)$ becomes a quasi-polynomial in $t$,
where $tP$ is the $t$th dilation of $P$
and $\#X$ denotes the cardinality of a finite set $X$.
This function is called the {\bf Ehrhart quasi-polynomial} of $P$ and we denote it by $\ehr_P$.
Let $P + \bm v =\{\bm x + \bm v \mid \bm x \in P\}$ be the convex polytope obtained from a convex polytope $P$ by the parallel translation by a vector $\bm v \in \mathbb R^d$.
The purpose of this paper is to develop a way to understand behaviors of $\ehr_{P+\bm v}$ when $\bm v $ runs over all vectors in $\mathbb Q^d$,
where $P$ is a fixed rational polytope.

One motivation of studying this problem is special behaviors of $\ehr_{P+\bm v}$
when we choose $\bm v \in \mathbb Q^d$ somewhat randomly.
Let us give an example to explain this. 
Let $T \subset \mathbb R^2$ be the trapezoid whose vertices are $(0,0),(1,0),(2,1)$ and $(0,1)$.
The Ehrhart quasi-polynomial of $T+ (\frac {17} {100}, \frac {52}{100})$ becomes the following quasi-polynomial having minimum period 100:
\begin{align*}
\ehr_{T+(\frac {17}{100},\frac {52}{100})}(t)
=
\begin{cases}
\frac 3 2 t^2 + \frac 5 2 t + 1 & (\ t \equiv 0\ ),\\
\frac 3 2 t^2 + \frac 3 2 t & (\ t \equiv 25,50,75\ ),\\
\frac 3 2 t^2 - \frac 1 2 t & \left( {\footnotesize 
\begin{array}{ll}
t \equiv 1,3,6,7,9,12,13,15,18,19,21,23,24,26,30, \vspace{-3pt}\\
\ 32,36,38,42,44,48,49,53,55,59,61,65,66,67,\vspace{-3pt}\\
\ 69,71,72,73,78,83,84,86,89,90,92,95,96,98\end{array}} \right),\\
\frac 3 2 t^2 + \frac 1 2 t & \left( {\footnotesize 
\begin{array}{ll}
t \equiv 2,4,5,8,10,11,14,16,17,20,22,28,29,31,33,34, \vspace{-3pt}\\
\ 35,37,40,41,45,47,51,52,54,56,57,58,60,62,64,68,\vspace{-3pt}\\
\ 70,74,76,77,79,80,81,82,85,87,88,91,93,94,97,99\end{array}} \right),
\end{cases}
\end{align*}
where ``$t \equiv a$" means ``$t \equiv a$ (mod $100$)".
This quasi-polynomial has several special properties.
For example, one can see
\begin{itemize}
    \item[($\alpha$)] It has a fairly large minimum period $100$, but it consists of only 4 polynomials.
    \item[($\beta$)] The polynomials $\frac 3 2 t^2 \pm \frac 1 2 t$ appear quite often comparing other two polynomials.
    \item[($\gamma$)] The polynomial $\frac 3 2 t^2 - \frac 1 2t$ appears when $t\equiv 1,3,6,7,\dots$, while the polynomial $\frac 3 2 t^2 +\frac 1 2 t=\frac 3 2 (-t)^2 - \frac 1 2(-t)$ appears when $t=\dots,93,94,97,99$. There seem to be a kind of reciprocity about the appearance of these two polynomials.
\end{itemize}
Our first goal is to explain why these phenomena occur by using a certain generalization of an Ehrhart quasi-polynomial, which was considered by McMullen \cite{McMullen} and is called a {\bf translated lattice point enumerators} in \cite{DY:2021}.

\subsection{First result}
We introduce a few notation to state our results.
A function $f : \mathbb Z \to \mathbb R$ is said to be a {\bf quasi-polynomial} if there is a natural number $q$ and polynomials $f_0,f_1,\dots,f_{q-1}$ such that
$$f(t)=f_i(t) \mbox{ for all }t \in \mathbb Z \mbox{ with } t \equiv i \mbox{ (mod }q).$$
A number $q$ is called a {\bf period} of $f$ and the polynomial $f_k$ is called the $k$th {\bf constituent} of $f$.
For convention, we define the $k$th constituent $f_k$ of $f$ for any $k \in \mathbb Z$ by setting $f_k=f_{k'}$ with $k' \equiv k$ (mod $q$).
For example, if $f$ has period $3$,
then the $7$th constituent equals the $1$st constituent $f_1$ and the $(-1)$th constituent equals the $2$nd constituent $f_2$.
We note that this definition does not depend on a choice of a period.
We will say that a function $L$ from $\mathbb Z_{> 0}$ (or $\mathbb Z_{ \geq 0}$) to $\mathbb R$ is a quasi-polynomial if there is a quasi-polynomial $f: \mathbb Z \to \mathbb R$ such that $L(t)=f(t)$ for all $t \in \mathbb Z_{> 0}$ (or $\mathbb Z_{ \geq 0}$),
and in that case we regard $L$ as a function from $\mathbb Z$ to $\mathbb R$ by identifying $L$ and $f$.

For a convex set $X \subset \mathbb R^d$ and a vector $\bm v \in \mathbb R^d$, we define the function $\TL_{X,\bm v}: \mathbb Z_{> 0} \to \mathbb R$ by
\[
\TL_{X,\bm v} (t)= \# \big( (tX+\bm v)\cap \mathbb Z^d \big)
\]
and call it the {\bf translated lattice points enumerator} of $X$ with respect to $\bm v$.
When $X$ is a convex polytope $P$,
we actually consider that $\TL_{P,\bm v}$ is a function from $\mathbb Z_{\geq 0}$ to $\mathbb R$ by considering that $tP=\{\bm 0\}$ when $t=0$.
Clearly $\TL_{P,\bm 0}$ is nothing but the Ehrhart quasi-polynomial of $P$.
Generalizing Ehrhart's results,
McMullen \cite[\S 4]{McMullen} proved that,
if $P$ is a rational polytope such that $qP$ is integral then $\TL_{P,\bm v}$ is a quasi-polynomial with period $q$,
and showed that there is a reciprocity between $\TL_{\mathrm{int}(P),\bm v}$ and $\TL_{P,-\bm v}$, where $\mathrm{int}(P)$ is the interior of $P$.
As we will see soon in Section 2, for a rational polytope $P \subset \mathbb R^d$ and $\bm v \in \mathbb Q^d$,
it follows from the above result of McMullen that
\begin{align}
\label{eq1-1}    
\mbox{the $k$th constituent of $\ehr_{P+\bm v}$}=
\mbox{the $k$th constituent of $\TL_{P,k\bm v}$}
\end{align}
for all $k \in \mathbb Z$.
This equation \eqref{eq1-1} was used in \cite{DY:2021} when $P$ is a lattice polytope, and is quite useful to study Ehrhart quasi-polynomials of translated polytopes.
Indeed, the equation says that knowing $\ehr_{P+\bm v}$ for all $\bm v \in \mathbb Q^d$ is essentially equivalent to knowing $\TL_{P,\bm v}$ for all $\bm v \in \mathbb Q^d$.
Our first goal is to explain that the latter information can be described as a finite information although $\ehr_{P+\bm v}$ could have arbitrary large minimum period.

To do this, we first discuss when
 $\TL_{P,\bm u}$ and $\TL_{P,\bm v}$ equal for different $\bm u,\bm v \in \mathbb R^d$ using toric arrangements.
For $\bm a=(a_1,\dots,a_d) \in \mathbb R^d$ and $b \in \mathbb R$,
let $H_{\bm a,b}$ be the hyperplane of $\mathbb R^d$ defined by the equation $a_1x_1+ \cdots+a_dx_d=b$.
Let $P$ be a rational convex $d$-polytope having $m$ facets $F_1,\dots,F_m$ such that each $F_k$ lies in the hyperplane $H_{\bm a_k,b_k}$ with $\bm a_k \in \mathbb Z^d$, $b_k \in \mathbb Z$ and $\mathrm{gcd}(\bm a_k,b_k)=1$.
We consider the arrangement of hyperplanes
$$\A_P=
\bigcup_{i =1}^m \{H_{\bm a_i,k} \mid  k \in \mathbb Z\}$$
and let $\Delta_P$ be the open polyhedral decomposition of $\mathbb R^d$ determined by $\A_P$.
Both $\A_P$ and $\Delta_P$ are invariant under translations by integer vectors,
so by the natural projection $\mathbb R^d \to \mathbb R^d /\mathbb Z^d$
they induce an arrangement of finite hyperplanes on the torus $\mathbb R^d /\mathbb Z^d$ 
and a finite open cell decomposition $\Delta_P/\mathbb Z^d$ of $\mathbb R^d /\mathbb Z^d$.
Let $[\bm x] \in \mathbb R^d/\mathbb Z^d$ denote the natural projection of $\bm x \in \mathbb R^d$ to $\mathbb R^d/\mathbb Z^d$.

\begin{theorem}
    \label{thm:1-2}
    With the notation as above,
    for $\bm u,\bm v \in \mathbb R^d$,
    if $[\bm u]$ and $[\bm v]$ belong to the same open cell of $\Delta_P /\mathbb Z^d$ then
    \[ \TL_{P,\bm u}(t)=\TL_{P,\bm v}(t) \ \ \mbox{ for all }t \in \mathbb Z_{\geq 0}.\]
\end{theorem}

For an open cell $C \in \Delta_P/\mathbb Z^d$,
define a quasi-polynomial $\TL_{P,C}$ by
\[
\TL_{P,C}=\TL_{P,\bm v} \ \ \ \mbox{ with }[\bm v ]\in C,
\]
which is well-defined by Theorem \ref{thm:1-2}.
Then \eqref{eq1-1} implies that the $k$th constituent of $\ehr_{P+\bm v}$ is 
the polynomial which appears as the $k$th constituent of $\TL_{P,C}$ with $[k \bm v] \in C$.
This provides us a way to compute $\ehr_{P+\bm v}$ for any $\bm v\in \mathbb Q^d$ from translated lattice points enumerators $\TL_{P,C}$.

Let us compute $\ehr_{T+(\frac {17} {100},\frac {52} {100})}(t)$ using this idea,
where $T$ is the trapezoid whose vertices are $(0,0),(1,0),(2,1)$ and $(0,1)$.
\begin{figure}
    \centering    \includegraphics[width=16cm,pagebox=cropbox]{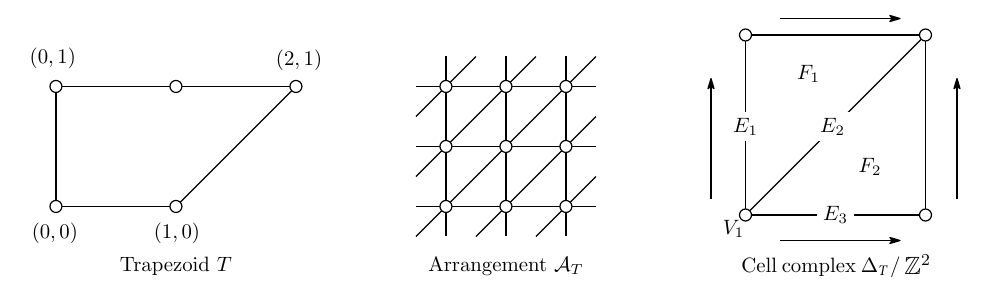}
    \caption{Trapezoid $T$, arrangement $\A_T$ and the cell complex $\Delta_T/\mathbb Z^2$ in $\mathbb R^2 /\mathbb Z^2$.}
    \label{trapezoid}
\end{figure}
Figure \ref{trapezoid} shows the arrangement $\A_T$ and the cell complex $\Delta_T/\mathbb Z^2$.
The complex $\Delta_T/\mathbb Z^2$ has two $2$-dimensional cells $F_1,F_2$,
three $1$-dimensional cells $E_1,E_2,E_3$ and one $0$-dimensional cell $V_1$ shown in Figure \ref{trapezoid}.
Since $T$ is a lattice polygon, each $\TL_{P,C}$ is a polynomial by McMullen's result, and here is a list of $\TL_{T,C}(t)$:
\begin{align} \label{eq1-2}
\begin{array}{lll}
    \TL_{T,F_1}(t)=\frac 3 2 t^2-\frac 1 2 t,\smallskip\\
    \TL_{T,F_2}(t)=\TL_{T,E_1}(t)=\TL_{T,E_2}(t)=\frac 3 2 t^2+\frac 1 2 t,\smallskip\\
    \TL_{T,E_3}(t)=\frac 3 2 t^2 + \frac 3 2 t,\smallskip\\
    \TL_{T,V_1}(t)= \frac 3 2 t^2 +\frac 3 2 t+1.
\end{array}
\end{align}
Also, for $k =0,1,2,\dots,99$, a computer calculation says
\begin{align} \label{eq1-3}
\left[ k \left( \frac {17} {100}, \frac {52} {100} \right) \right]
\in\begin{cases}
    V_1 & (\ k \equiv 0\ ),\\
    E_3 & (\ k \equiv 25,50,75\ ),\\
    E_2 & (\ k \equiv 20,40,60,80\ ),\\
    F_1 & 
    \left( {\footnotesize 
\begin{array}{ll}
k \equiv 1,3,6,7,9,12,13,15,18,19,21,23,24,26,30, \vspace{-3pt}\\
\ 32,36,38,42,44,48,49,53,55,59,61,65,66,67,\vspace{-3pt}\\
\ 69,71,72,73,78,83,84,86,89,90,92,95,96,98\end{array}} \right),\\
F_2 & \left( {\footnotesize 
\begin{array}{ll}
k \equiv 2,4,5,8,10,11,14,16,17,22,28,29,31,33,34, \vspace{-3pt}\\
\ 35,37,41,45,47,51,52,54,56,57,58,62,64,68,70,\vspace{-3pt}\\
\ 74,76,77,79,81,82,85,87,88,91,93,94,97,99\end{array}} \right).
\end{cases}
\end{align}
Since \eqref{eq1-1} says that the $k$th constituent of $\ehr_{P+\bm v}$ equals the $k$th constituent of $\TL_{P,k\bm v}$,
which equals $\TL_{P,C}$ with $[k \bm v] \in C \in \Delta_P/\mathbb Z^d$,
the equations \eqref{eq1-2} and \eqref{eq1-3} recover the formula of  $\ehr_{T+(\frac {17}{100}, \frac {52}{100})}(t)$ given at the beginning of this section.

As we will see,
the proof of Theorem \ref{thm:1-2} is somewhat straightforward,
and the way of computing $\ehr_{P+\bm v}(t)$ from $\TL_{P,C}(t)$ explained above may be considered as a kind of an observation rather than a new result.
But we think that this is a useful observation.
For example,
this way allows us to visualize how the constituents of $\ehr_{P+\bm v}$ change by plotting the points $[k \bm v]$ on $\mathbb R^d/\mathbb Z^d$.
Also, we can see why properties ($\alpha$), ($\beta$) and ($\gamma$) occur from this observation.
For the property $(\alpha)$,
we only see 4 polynomials in $\ehr_{T+(\frac {17}{100},\frac {52}{100})}$ because we have only 4 types of translated lattice point enumerators.
More generally, it can be shown that, if we fix a rational polytope $P$, then we can only have a finite number of polynomials as constituents of $\ehr_{P+\bm v}$ (Theorem \ref{finiteness}).
For the property ($\beta$),
the polynomials $\frac 3 2 t^2 \pm \frac 1 2 t$ appear many times simply because they are polynomials assigned to maximal dimensional cells of $\Delta_T/\mathbb Z^2$
(indeed, if we choose $\bm v$ randomly, then $[k\bm v]$ is likely to belong to a maximal dimensional cell).
Finally, we will see in Section 5 that the property ($\gamma$) can be figured out from the reciprocity of $\TL_{P,\bm v}$
(see Corollary \ref{cor:TLrecipro}).

\subsection{Second Result}
Recently real-valued extension of Ehrhart functions,
namely, the function $\ehr_{P}^{\mathbb R}:\mathbb R_{\geq 0} \to \mathbb Z_{\geq 0}$ given by $\ehr_{P}^{\mathbb R}(t)=\#(tP \cap \mathbb Z^d)$ for all $t \in \mathbb R_{\geq  0}$,
catch interests \cite{BBKV,BER,Linke,Royer1,Royer2}.
One surprising result on this topic is the following result of Royer \cite{Royer1,Royer2} proving that $\ehr_{P+ \bm v}^{\mathbb R}$ for all $\bm v \in \mathbb Z^d$ determines the polytope $P$.

\begin{theorem}[Royer]
\label{thm:Royer}
Let $P$ and $Q$ be rational polytopes in $\mathbb R^d$.
If $\ehr_{P+\bm v}^{\mathbb R}(t)=\ehr_{Q+\bm v}^{\mathbb R}(t)$
for all $\bm v \in \mathbb Z^d$ and $t \in \mathbb R_{\geq  0}$,
then $P=Q$.
\end{theorem}

Our second result is somewhat analogous to this result of Royer.
We prove that $\ehr_{P+\bm v}$ for all $\bm v \in \mathbb Q^d$ determines the polytope $P$ up to translations by integer vectors.

\begin{theorem}
\label{thm:1-4}
Let $P$ and $Q$ be rational $d$-polytopes in $\mathbb R^d$.
If $\ehr_{P+\bm v}(t)=\ehr_{Q+\bm v}(t)$
for all $\bm v \in \mathbb Q^d$ and all $t \in \mathbb Z_{\geq 0}$,
then $P=Q+\bm u$ for some $\bm u \in \mathbb Z^d$.
\end{theorem}

After we submitted the paper, we realized that Theorem \ref{thm:1-4} is not new and appears in the thesis of Alhajjar  in a bit stronger form \cite[Theorem 3.9]{alh}. But we keep the proof of Theorem \ref{thm:1-4} since we feel that our proof is more precise and some argument in the proof is also used to prove the third result.

\subsection{Third result}
The original motivation of this study actually comes from an attempt to generalize results of de Vries and Yoshinaga in \cite{DY:2021},
who found a connection between symmetries on constituents of $\ehr_{P+\bm v}$ and geometric symmetries of $P$.
Indeed, the following result is one of the main results in \cite{DY:2021}.
We say that a quasi-polynomial $f$ is {\bf symmetric} if the $k$th constituent of $f$ equals the $(-k)$th constituent of $f$ for all $k \in \mathbb Z$.
Also, a convex polytope $P \subset \mathbb R^d$ is said to be {\bf centrally symmetric}  if $P=-P+\bm x$ for some $\bm x \in \mathbb R^d$.

\begin{theorem}[de Vries--Yoshinaga]
\label{thm:deVriesYoshinaga}
Let $P \subset \mathbb R^d$ be a lattice $d$-polytope.
The following conditions are equivalent.
\begin{itemize}
    \item[(1)] $\ehr_{P+\bm v}$ is symmetric for all $\bm v \in \mathbb Q^d$.
    \item[(2)] $P$ is centrally symmetric. 
\end{itemize}
\end{theorem}

As posed in \cite[Problem 6.7]{DY:2021},
it is natural to ask if there is a generalization of this result for rational polytopes.
Theorem \ref{thm:deVriesYoshinaga} actually proves that,
if a rational polytope $P$ satisfies the property (1) of the above theorem, then $P$ must be centrally symmetric (see Corollary \ref{projsym}).
We generalize Theorem \ref{thm:deVriesYoshinaga} in the following form.

\begin{theorem}
    \label{thm:1-6}
Let $P \subset \mathbb R^d$ be a rational $d$-polytope.
The following conditions are equivalent.
\begin{itemize}
    \item[(1)] $\ehr_{P+\bm v}$ is symmetric for all $\bm v \in \mathbb Q^d$.
    \item[(2)] $P$ is centrally symmetric and $2(P-\bm c)$ is integral, where $\bm c$ is the center of symmetry of $P$. 
\end{itemize}
\end{theorem}

\subsection*{Organization of the paper}
This paper is organized as follows:
We first quickly review basic known properties of Ehrhart quasi-polynomials and translated lattice points enumerators in Section 2.
In Section 3,
we study translated lattice points enumerators using arrangement $\A_P$ and prove Theorem \ref{thm:1-2}.
Then, after seeing two examples in Section 4,
we discuss a reciprocity of translated lattice points enumerators on maximal cells of $\Delta_P/\mathbb Z^d$ in Section 5.
In Section 6, we prove that translated lattice point enumerators determine the polytope $P$ up to translations by integer vectors.
In Section 7, we study translated lattice points enumerators of polytopes with some symmetry,
in particular, prove Theorem \ref{thm:1-6}.
In Section 8,
we discuss a connection to commutative algebra,
more precisely, we discuss a connection between translated lattice points enumerators and conic divisorial ideals in Ehrhart rings.
We list a few problems which we cannot solve in the last section 9.

\subsection*{Acknowledgements}
We thank Katharina Jochemko for letting us know McMullen's work in \cite{McMullen}, 
Matthias Beck for letting us know the work of Royer in \cite{Royer1,Royer2}, 
and Elie Alhajjar for informing us his Ph.D. Thesis \cite{alh} which contains 
closely related results. 
We also thank the anonymous
reviewers for helpful comments and suggestions.
The first author is partially supported by KAKENHI 20K03513 and 21KK0043,
the second author is partially supported by KAKENHI 21K0319 and 21H00975,
and the third author is partially supported by KAKENHI 18H01115 and 23H00081.

\section{Ehrhart quasi-polynomials and translated lattice point enumerators}

In this section,
we recall basic results on Ehrhart quasi-polynomials and explain a connection between Ehrhart quasi-polynomials of translated polytopes and translated lattice point enumerators.

\subsection{Ehrhart quasi-polynomial}
We quickly recall Ehrhart's theorems.
We refer the readers to \cite{BR:2007,Gru:book,Z} for basics on convex polytopes.
A {\bf convex polytope} $P$ in $\mathbb R^d$ 
is a convex hull of finitely many points in $\mathbb R^d$.
The {\bf dimension} of a polytope $P$ is the dimension of its affine hull.
A $k$-dimensional convex polytope will be simply called a {\bf $k$-polytope} in this paper.
A convex polytope $P$ is said to be {\bf integral} (resp.\ {\bf rational}) if all the vertices of $P$ are lattice points (resp.\ rational points).
The \textbf{denominator} of a rational polytope $P$ is the smallest integer $k>0$ such that $kP$ is integral.
The following result is a fundamental result in Ehrhart theory. See \cite[Theorems 3.23 and 4.1]{BR:2007}.

\begin{theorem}[Ehrhart]
Let $P \subset \mathbb R^d$ be a rational polytope and $q$ the denominator of $P$.
Then the function $\ehr_P: \mathbb Z_{\geq 0} \to \mathbb R$ defined by
\[
\ehr_P(t)= \# (t P \cap \mathbb Z^d)
\]
is a quasi-polynomial with period $q$.
\end{theorem}

As we noted in the Introduction,
we regard $\ehr_P$ as a function from $\mathbb Z$ to $\mathbb R$ by identifying it with the corresponding quasi-polynomial $f: \mathbb Z \to \mathbb R$ that coincides with $\ehr_P$ on $\mathbb Z_{> 0}$.
Thus, if $q$ is a period of $f$, then for a positive integer $t>0$ we set $\ehr_P(-t)=f_k(-t)$, where $f_k$ is the $k$th constituent of $f$ with $-t \equiv k$ (mod $q$).
The quasi-polynomial $\ehr_P$ is called the {\bf Ehrhart quasi-polynomial} of $P$.

The following reciprocity result is another important result in Ehrhart theory.

\begin{theorem}[Ehrhart reciprocity]
Let $P \subset \mathbb R^d$ be a rational $d$-polytope.
Then 
$$\# \big(\mathrm{int}(tP)\cap \mathbb Z^d \big) = (-1)^d \ehr_P(-t) \ \ \mbox{ for }\  t \in \mathbb Z_{>0}.$$
\end{theorem}

\subsection{Translated lattice points enumerator}
Recall that,
for a convex set $X \subset \mathbb R^d$ and $\bm v \in \mathbb R^d$,
the {\bf translated lattice points enumerator} of $X$ w.r.t.\ $\bm v$ is the function
$\TL_{X,\bm v}$ defined by
\begin{align}
\label{def1}
\TL_{X,\bm v}(t)=\# \big( (tX + \bm v) \cap \mathbb Z^d \big)
\ \ \mbox{ for }t \in \mathbb Z_{> 0}.
\end{align}
When $X$ is a polytope $P$,
we actually consider that $\TL_{P,\bm v}$ is a function from $\mathbb Z_{\geq 0}$ to $\mathbb R$ by
setting
$\TL_{P,\bm v}(0)=\# (\{\bm v\} \cap \mathbb Z^{d})$.
Thus $\TL_{P,\bm v}(0)=1$ if $\bm v \in \mathbb Z^d$ and
$\TL_{P,\bm v}(0)=0$ if $\bm v \not \in \mathbb Z^{d}$.
In this way, we may consider that $\TL_{P,\bm v}$ is a counting function of lattice points in the translated cone.
Indeed,
for a convex polytope $P \subset \mathbb R^d$, if we write $\mathcal C_P$ for the cone generated by $\{(\bm x,1)\mid \bm x \in P\}$,
Then we have
\[
\TL_{P,\bm v}(t)=\# \big(\big( \mathcal C_P +(\bm v,0) \big)\cap H_{x_{d+1}=t} \big)\cap \mathbb Z^{d+1} \ \ \mbox{ for } t \geq 0,\]
where $H_{x_{d+1}=t}=\{(x_1,\dots,x_{d+1})\mid x_{d+1}=0\}$.

McMullen \cite[\S 4]{McMullen} proved the following generalization of Ehrhart's results.
(We will give an algebraic proof of this theorem later in section 8.)

\begin{theorem}[McMullen]
\label{thm:McMullen}
Let $P \subset \mathbb R^d$ be a rational $d$-polytope and $q$ the denominator of $P$.
Then
\begin{itemize}
    \item[(1)] For any $\bm v \in \mathbb R^d$,
    the function $\TL_{P,\bm v}$ is a quasi-polynomial with period $q$.
    \item[(2)] For any $\bm v \in \mathbb R^d$, one has
    \[
    \TL_{\mathrm{int}(P),\bm v}(t)=(-1)^d \TL_{P,-\bm v}(-t)
    \ \ \mbox{ for $t \in \mathbb Z_{>0}$.}\]
\end{itemize}
\end{theorem}
We remark that $\bm v$ is not necessarily a rational point in the above theorem.
Also the theorem says that, if $P$ is integral, then the function $\TL_{P,\bm v}$ is a polynomial.

The following connection between Ehrhart quasi-polynomials of translated polytopes and translated lattice points enumerators, which essentially appeared in \cite[Corollary 3.4]{DY:2021}, is fundamental in the rest of this paper.

\begin{lemma}
\label{ehrToTrans}
Let $P \subset \mathbb R^d$ be a rational $d$-polytope
and $\bm v \in \mathbb Q^d$.
For all $k \in \mathbb Z$,
one has
\[
\mbox{\rm the $k$th constituent of $\ehr_{P+\bm v}$}=
\mbox{\rm the $k$th constituent of $\TL_{P,k\bm v}$}.
\]
\end{lemma}

\begin{proof}
We may assume $k\geq 0$.
Let $\rho$ and $\rho'$ be positive integers such that $\rho P$ is integral and $\rho' \bm v \in \mathbb Z^d$.
Let $q$ be a common multiple of $\rho$ and $\rho'$.
Then $q$ is a common period of quasi-polynomials $\ehr_{P+ \bm v}$ and $\TL_{P,\bm v}$.
For every integer $t \geq 0$ with $t \equiv k$ (mod $q$) we have\begin{align*}
    \ehr_{P+\bm v}(t)&=\#\big((t P+t \bm v) \cap \mathbb Z^d\big)
    =\# \big( (t P + k \bm v) \cap \mathbb Z^d \big)
    =\TL_{P,k\bm v}(t),
\end{align*}
where the second equality follows from $(t-k) \bm v \in \mathbb Z^d$.
Since both $\ehr_{P+\bm v}$ and $\TL_{P,\bm v}$ are quasi-polynomials with a period $q$, the above equation proves the desired property.
\end{proof}

\begin{remark}
Let $P \subset \mathbb R^d$ be a rational $d$-polytope and $\bm v \in \mathbb R^d$.
Like usual Ehrhart quasi-polynomials,
each constituent of $\TL_{P,\bm v}$ is a polynomial of degree $d$ whose leading coefficient equals the volume of $P$.
Indeed, if $f_k$ is the $k$th constituent of $\TL_{P,\bm v}$ and $q$ is a period of $\TL_{P,\bm v}$, then $\lim_{t \to \infty} \frac {f_k(qt+k)} {(qt+k)^d}=
\lim_{t \to \infty} \frac {\#((qt+k)P \cap \mathbb Z^d)} {(qt+k)^d}$ is the volume of $P$.
Since $f_k$ is a polynomial,
this means that $f_k$ has degree $d$ and the coefficient of $t^d$ in $f_k$ equals the volume of $P$.
\end{remark}

\section{Translated lattice points enumerators and toric arrangements}

In this section,
we study when $\TL_{P,\bm v}$ equals $\TL_{P,\bm u}$ for different $\bm v, \bm u \in \mathbb R^d$ using toric arrangements.
Recall that the translated lattice points enumerator $\TL_{P,\bm v}$ can be identified with a generating function of a translated cone $\mathcal C_P+(\bm v,0)$
because of the equality
\begin{align}
\label{useTHM3.8}
\sum_{(a_1,\dots,a_{d+1}) \in (\mathcal C_P+(\bm v,0))\cap \mathbb Z^{d+1}} z^{a_{d+1}} = \sum_{t=0}^\infty \big(\TL_{P,\bm v}(t) \big) z^t.
\end{align}
This in particular says that if $\mathcal C_P+(\bm u,0)$ and $\mathcal C_P+(\bm v,0)$ have the same lattice points, then we have $\TL_{P,\bm u}=\TL_{P,\bm v}$.
To prove Theorem \ref{thm:1-2},
we mainly study when $\mathcal C_P+(\bm u,0)$ and $\mathcal C_P+(\bm v,0)$ have the same lattice points.

We note that such a study is not very new.
Indeed, lattice points in the translated cone $\mathcal C_P+\bm v$ is  closely related to conic divisorial ideals of Ehrhart rings studied in \cite{Bruns,BG},
and Bruns \cite{Bruns} explains for which $\bm u,\bm v \in \mathbb R^{d+1}$ the lattice points in $\mathcal C_P+ \bm u$ equal those in $\mathcal C_P+ \bm v$.
We will explain this connection to commutative algebra later in Section 8.

\subsection{Regions associated with hyperplane arrangements}
We first introduce some notation on arrangements of hyperplanes.
For $\bm x,\bm y \in \mathbb R^d$, we write $(\bm x,\bm y)$ for the standard inner product.
Also, for $\bm a \in \mathbb R^d \setminus \{\bm 0\}$ and $b \in \mathbb R$,
we write
$$H^{\geq }_{\bm a,b}=\{ \bm x \in \mathbb R^d \mid (\bm a, \bm x)\geq b\}\ \
\mbox{ and } \ \ H^{>}_{\bm a,b}=\{ \bm x \in \mathbb R^d \mid (\bm a, \bm x)> b\}$$
for closed and open half space defined by the linear inequalities $(\bm a,\bm x)\geq b$ and $(\bm a,\bm x)>b$, respectively,
and write
$$H_{\bm a,b}=\{ \bm x \in \mathbb R^d \mid (\bm a, \bm x)=b\}$$
for the hyperplane defined by the linear equation $(\bm a,\bm x)=b$.
In the case where $\bm a$ can be chosen from $\mathbb Z^d$ and $b$ is from $\mathbb Z$, we call the hyperplane $H_{\bm a,b}$ \textit{rational}. 
Let $N=\{\bm a_1,\dots,\bm a_m \}$ be a set of elements in $\mathbb Z^{d} \setminus \{\bm 0\}$.
Define the arrangement of hyperplanes
$$\A_N=\{H_{\bm a,k} \mid \bm a \in N,\  k \in \mathbb Z\}.$$
See Figure \ref{fig1}.
From now on, we fix an order $\bm a_1,\dots,\bm a_m$ of elements of $N$.
We define the map $\varphi_{(\bm a_1,\dots,\bm a_m)}: \mathbb R^d \to \mathbb R^m$ by
$$\varphi_{(\bm a_1,\dots,\bm a_m)}(\bm x)=\big((\bm a_1,\bm x),(\bm a_2,\bm x),\dots,(\bm a_m,\bm x)\big).$$
For $x \in \mathbb R$,
we write $\lfloor x \rfloor =\max\{ \ell \in \mathbb Z \mid \ell \leq x\}$
and
$\lceil x \rceil =\min\{ \ell \in \mathbb Z \mid \ell \geq x\}$.
Also, given an integer sequence $\bm c=(c_1,\dots,c_m) \in \mathbb Z^m$,
we define 
$$U_{\bm c}^N
=\{ \bm x \in \mathbb R^d \mid \lceil \varphi_{(\bm a_1,\dots,\bm a_m)}(\bm x) \rceil =\bm c\}
=\{ \bm x \in \mathbb R^d \mid c_{i}-1 < (\bm a_i,\bm x) \leq c_i \mbox{ for }i=1,2,\dots,m\}$$
where $\lceil (x_1,\dots,x_m)\rceil =(\lceil x_1 \rceil, \dots,\lceil x_m \rceil)$.
We call $U_{\bm c}^N$ an {\bf upper region} of $N$.
Note that $U_{\bm c}^N$ could be empty.
Also we have the partition
\[
\textstyle
\mathbb R^d = \bigsqcup_{\bm c \in \mathbb Z^d} U_{\bm c}^N\]
where $\bigsqcup$ denotes a disjoint union.
We write $\Lambda_N$ for the set of all upper regions of $N$.
The set $\Lambda_N$ is stable by translations by integer vectors,
so $\mathbb Z^d$ acts on these sets.
Indeed, since $\bm a_1,\dots,\bm a_m$ are integer vectors,
for any $\intvec \in \mathbb Z^d$, we have
$$U_{\bm c}^N+ \intvec= U_{\bm c+\varphi_{(\bm a_1,\dots,\bm a_m)}(\intvec)}^N.
$$
We write $\Lambda_N/\mathbb Z^d$
for the quotient of these sets by this $\mathbb Z^d$-action defined by translations by integer vectors.
This set can be considered as a partition of the $d$-torus $\mathbb R^d /\mathbb Z^d$.

\begin{example}
\label{ex:4.1}
Let $N=\{(1,0),(-1,2)\}$.
Then the set $\Lambda_N/\mathbb Z^2$ consists of two elements
with the following representatives:
\begin{align*}
    R_1 & = U_{(1,1)}^N=\{(x,y) \in \mathbb R^2 \mid 0<x \leq 1, 0 < -x+2y \leq 1\},\\
    R_2 & = U_{(1,0)}^N=\{(x,y) \in \mathbb R^2 \mid 0<x \leq 1, -1 < -x+2y \leq 0\}.
\end{align*}
See Figure \ref{fig1} for the visualization of $\A_N$ and $\Lambda_N /\mathbb Z^2$.
\end{example}

\begin{figure}
\includegraphics[width=8cm,pagebox=cropbox]{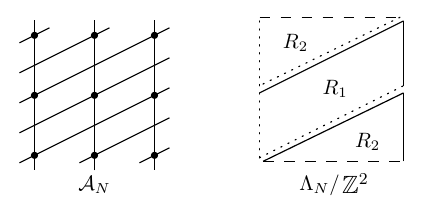}
\caption{$\A_N$ and $\Lambda_N/\mathbb Z^2$ when $N=\{(1,-2),(0,1)\}$.
Doted and solid lines are open and closed boundaries respectively. Dashed lines indicate the occurrence of identification inside a region.
}
\label{fig1}
\end{figure}

\subsection{Upper regions and lattice points in translated cones.}
We now explain how upper regions relate to lattice points in translated cones.
We first recall two basic facts on lattice points.
The following lemma is an easy consequence of Euclidian algorithm.

\begin{lemma}
\label{lem:gcd}
Let $\bm a \in \mathbb Z ^d\setminus \{\bm 0\}$, $b \in \mathbb R$ and $g=\mathrm{gcd}(\bm a)$.
A linear equation $(\bm a,\bm x)=b$ has an integral solution if and only if $b \in g\mathbb Z$.
\end{lemma}

We also need the following statement which easily follows from Lemma \ref{lem:gcd}.
\begin{lemma}
    \label{gcdcor}
    Let $H \subset \mathbb R^d$ be a rational hyperplane and let $\bm v \in \mathbb R^d$ be a point such that $H + \bm v \ne H$.
    There is an $\varepsilon>0$ such that $H+s\bm v$ contains no lattice points for any $0<s\leq \varepsilon$.
\end{lemma}

\begin{lemma}
\label{lem:cone}
Let $H \subset \mathbb R^d$ be a rational hyperplane.
Any $(d-1)$-dimensional convex cone in $H$ contains a lattice point. 
\end{lemma}

\begin{proof}
Let $H=H_{\bm a,b}$ for some $\bm a \in \mathbb Z^d$ and $b \in \mathbb Z$. Without loss of generality, we may assume $b=0$. 
Since any $d$-dimensional convex cone in $\mathbb R^d$ contains a lattice point,
the lemma follows from the fact that $H \cap \mathbb Z^d \cong \mathbb Z^{d-1}$ as $\mathbb Z$-modules.
\end{proof}

Let $P \subset \mathbb R^d$ be a rational $d$-polytope.
By the Weyl-Minkowski theorem for convex cones,
the cone $\mathcal C_P$ has the unique presentation
\begin{align}
    \label{4-1}
    \mathcal C_P=H_{\bm a_1,0}^\geq \cap \cdots \cap H_{\bm a_m,0}^\geq
\end{align}
such that
\begin{itemize}
    \item[(1)] each $\bm a_i$ is primitive (that is, $\mathrm{gcd}(\bm a_i)=1$), and
    \item[(2)] the presentation is irredundant, that is, each $H_{\bm a_i,0}^\geq$ cannot be omitted from the presentation.
\end{itemize}
Note that the second condition says that $\mathcal C_P \cap H_{\bm a_i,0}$ is a facet of $\mathcal C_P$.
The vectors $\bm a_1,\dots,\bm a_m$ in \eqref{4-1} are called (inner) {\bf normal vectors} of $\mathcal C_P$
and we write $\widetilde N(P)=\{\bm a_1,\dots,\bm a_m\}$ for the set of all normal vectors of $\mathcal C_P$.
The next statement was given in \cite{Bruns}

\begin{proposition}[Bruns]
\label{idealregion}
Let $P \subset \mathbb R^d$ a convex polytope and
$\widetilde N(P)=\{\bm a_1,\dots,\bm a_m\}$.
Let $\bm u,\bm v \in \mathbb R^{d+1}$.
The following conditions are equivalent.
\begin{itemize}
    \item[(1)] $(\mathcal C_P+ \bm u)\cap \mathbb Z^{d+1}=(\mathcal C_P +\bm v) \cap \mathbb Z^{d+1}$.
    \item[(2)] $\lceil \varphi_{(\bm a_1,\dots,\bm a_m)}(\bm u)\rceil=\lceil \varphi_{(\bm a_1,\dots,\bm a_m)}(\bm v)\rceil$, that is, $\bm u$ and $\bm v$ belong to the same upper region of $\Lambda_{\widetilde N(P)}$.
\end{itemize}
\end{proposition}

\begin{proof}
Let $\lceil \varphi_{(\bm a_1,\dots,\bm a_m)}(\bm u) \rceil=(c_1,\dots,c_m)$ and let $\lceil \varphi_{(\bm a_1,\dots,\bm a_m)}(\bm v) \rceil =(d_1,\dots,d_m)$.
Since 
$$\mathcal C_P+\bm u=H_{\bm a_1,(\bm a_1,\bm u)}^\geq \cap \cdots \cap H_{\bm a_m,(\bm a_m,\bm u)}^{\geq} $$
and since Lemma \ref{lem:gcd} says
$$H_{\bm a,b} ^\geq\cap \mathbb Z^{d+1}=H_{\bm a,\lceil b \rceil}^{\geq} \cap \mathbb Z^{d+1} \ \ \mbox{ for any } \bm a \in \mathbb Z^d,\ b \in \mathbb R, $$
we have
\begin{align}
    \label{eqcone}
(\mathcal C_P+\bm u) \cap \mathbb Z^{d+1}= \big(H_{\bm a_1,c_1}^\geq \cap \cdots \cap H_{\bm a_m,c_m}^{\geq}\big)\cap \mathbb Z^{d+1}
\end{align}
and
$$(\mathcal C_P+\bm v) \cap \mathbb Z^{d+1}= \big(H_{\bm a_1,d_1}^\geq \cap \cdots \cap H_{\bm a_m,d_m}^{\geq}\big)\cap \mathbb Z^{d+1}.$$
These prove (2) $\Rightarrow$ (1).

We prove (1) $\Rightarrow$ (2).
We assume $c_1 <d_1$ and prove $(\mathcal C_P+ \bm u) \cap \mathbb Z^{d+1} \ne (\mathcal C_P + \bm v)\cap \mathbb Z^{d+1}$.
In this case $F=(\mathcal C_P + \bm u)\cap H_{\bm a_1,c_1}$ contains a $d$-dimensional cone in $H_{\bm a_1,c_1}$, so it contains a lattice point by Lemma \ref{lem:cone}.
On the other hand,
since $c_1<d_1$ we have
$\mathbb Z^{d+1} \cap (\mathcal C_P+\bm v) \cap H_{\bm a_1,c_1}=\varnothing$. These prove $(\mathcal C_P+\bm u) \cap \mathbb Z^{d+1} \ne (\mathcal C_P + \bm v)\cap \mathbb Z^{d+1}$.
\end{proof}


\begin{remark}
If $N=\widetilde N(P)$ for some rational $d$-polytope $P \subset \mathbb R^d$
(that is, $N$ is the set of normal vectors of the cone $\mathcal C_P \subset \mathbb R^{d+1}$),
then the set $\Lambda_N$ has a special property that every element of $\Lambda_N$ has dimension $d+1$.
Indeed, if $R \in \Lambda_N$ and $\bm x \in R$
then we have $\bm x - \bm y \in R$ for all $\bm y \in \mathrm{int}(\mathcal C_P)$ that is sufficiently close to the origin,
which implies that $R$ has dimension $d+1$.
As we see in Example \ref{ex4.2},
this property does not hold when $N$ is the set of normal vectors of a polytope (not the cone over a polytope).
\end{remark}




We have studied lattice points in translated cones $\mathcal C_P + \bm v$,
but we are actually interested in a special case when $\bm v=(\bm v',0)$ since this is the case which is related to translated lattice points enumerators.
Below we describe when $\mathcal C_P+(\bm u,0)$ and $\mathcal C_P+(\bm v,0)$ have the same lattice points.
Let $P \subset \mathbb R^d$ be a rational $d$-polytope.
By the fundamental theorem on convex polytopes,
there is the unique presentation
$$P=H_{\bm a_1,b_1} ^\geq \cap \cdots \cap H_{\bm a_m,b_m}^\geq$$
such that
\begin{itemize}
    \item[(1)] each $(\bm a_i,b_i) \in \mathbb Z^{d+1}$ is primitive, and
    \item[(2)] the presentation is irredundant.
\end{itemize}
The vectors $\bm a_1,\dots,\bm a_m$ are called \textbf{normal vectors} of $P$.
We define 
$$N(P)=\{\bm a_1,\dots,\bm a_m\}.$$
We note that 
$$\widetilde N(P)=\{(\bm
 a_1,b_1),\dots,(\bm a_m,b_m)\}$$
since if $H \subset \mathbb R^d$ is a hyperplane defined by $a_1x_1+ \cdots + a_d x_d=b$ then the cone $\mathcal C_H$ is the hyperplane defined by $a_1x_1+ \cdots + a_d x_d=bx_{d+1}$.
We write
\[ \A_P=\A_{N(P)}\ \  \mbox{ and } \ \ \Lambda_P=\Lambda_{N(P)}.\]
The following statement is essentially a consequence of Proposition \ref{idealregion}.

\begin{corollary}
\label{stdidealregion}
Let $P \subset \mathbb R^d$ be a rational $d$-polytope and let
$\bm u,\bm v \in \mathbb R^d$. The following conditions are equivalent.
\begin{itemize}
    \item[(1)] $(\mathcal C_P +(\bm u,0)) \cap \mathbb Z^{d+1}=(\mathcal C_P +(\bm v,0)) \cap \mathbb Z^{d+1}$.
    \item[(2)] $(\bm u,0)$ and $(\bm v,0)$ belong to the same upper region in $\Lambda_{\widetilde N(P)}$.
    \item[(3)] $\bm u$ and $\bm v$ belong to the same upper region in $\Lambda_P$
\end{itemize}
\end{corollary}

\begin{proof}
The equivalence between (1) and (2) is Proposition \ref{idealregion}.
Let $(\bm a_1,\dots,\bm a_m)$ be the sequence of normal vectors of $P$ and let $((\bm a_1,b_1),\dots,(\bm a_m,b_m))$ be that of $\mathcal C_P$.
The equivalence between (2) and (3) follows from the fact that $\varphi_{(\bm a_1,\dots,\bm a_m)}(\bm x)=\varphi_{((\bm a_1,b_1),\dots,(\bm a_m,b_m))}(\bm x,0)$ for all $\bm x \in \mathbb R^{d}$.
\end{proof}

We now discuss a consequence of Corollary \ref{stdidealregion} to translated lattice point enumerators
and Ehrhart quasi-polynomials.
Recall that
$[\bm x]$ denotes the image of $\bm x \in \mathbb R^d$ by the natural projection $\mathbb R^d \to \mathbb R^d/\mathbb Z^d$.

\begin{theorem}
\label{main:tlpregion}
Let $P \subset \mathbb R^d$ be a rational $d$-polytope
and let $\bm u,\bm v \in \mathbb R^d.$
\begin{itemize}
    \item[(1)] If $[\bm u]$ and $[\bm v]$ belong to the same region in $\Lambda_P/\mathbb Z^d$, then $\TL_{P,\bm u}(t)=\TL_{P,\bm v}(t)$ for all $t \in \mathbb Z_{\geq 0}$.
    \item[(2)] The set $\{\TL_{P,\bm w} \mid \bm w \in \mathbb R^d\}$ is a finite set.
\end{itemize}
\end{theorem}

\begin{proof}
(1) Corollary \ref{stdidealregion} says that if $[\bm u]$ and $[\bm v]$ belong to the same region in $\Lambda_P/\mathbb Z^d$,
then 
\[
\big(\mathcal C_P+(\bm u,0) \big)\cap \mathbb Z^{d+1} =\big(\mathcal C_P+(\bm v,0)\big)\cap \mathbb Z^{d+1} +(\intvec,0),\]
where $\intvec \in \mathbb Z^{d}$ is the vector such that $\bm u$ and $\bm v+ \intvec$ belong to the same region of $\Lambda_P$.
Then \eqref{useTHM3.8} implies $\TL_{P,\bm u}(t)=\TL_{P,\bm v}(t)$ for all $t \in \mathbb Z_{\geq 0}$.

(2) 
Since there are only finitely many hyperplanes in $\mathcal A_P$ that intersect $[0,1)^d$,
the set $\{R \in \Lambda_P \mid R \cap [0,1)^d \ne \varnothing \}$ is finite,
which implies that
$\Lambda_{P}/\mathbb Z^d$ is a finite set.
This fact and (1) prove the desired statement.
\end{proof}

\begin{example}
\label{ex4.2}
Consider the trapezoid $T$ from the Introduction.
The set of normal vectors of $T$ is $N(T)=\{(1,0),(0,1),(0,-1),(-1,1)\}$.
Then the set of upper regions $\Lambda_T/\mathbb Z^2$ consists of 4 elements with the following representatives:
{\small
\begin{align*}
    V & = U_{(0,0,0,0)}=\{(x,y) \in \mathbb R^2 \mid -1< x \leq 0,\ -1 < y \leq 0,\ -1 <-y \leq 0,\ -1 < -x+y \leq 0\},\\
    E & = U_{(1,0,0,0)}=\{(x,y) \in \mathbb R^2 \mid 0< x \leq 1,\ -1 < y \leq 0,\ -1 < -y \leq 0,\ -1 < -x+y \leq 0\},\\
    R_1 & = U_{(1,1,0,0)}=\{(x,y) \in \mathbb R^2 \mid 0< x \leq 1,\ 0 < y \leq 1,\ -1<-y \leq 0,\ -1 < -x+y \leq 0\},\\
    R_2 & = U_{(1,1,0,1)}=\{(x,y) \in \mathbb R^2 \mid 0< x \leq 1,\ 0 < y \leq 1,\ -1 < -y \leq 0,\ 0<-x+y\leq 1\}.
\end{align*}
}

\noindent
See Figure \ref{fig3} for the visualization of $\Lambda_T/\mathbb Z^2$ in the torus $\mathbb R^2 /\mathbb Z^2$.
Note that $V$ is a one point set.
Theorem \ref{main:tlpregion} says that $\TL_{T,\bm v}$ only depends on the upper region in $\Lambda_T/\mathbb Z^2$ where $[\bm v]$ belongs,
and the table below is a list of the polynomials $\TL_{T,C}(t)$ in each upper region $C \in \Lambda_T/\mathbb Z^2$.
\bigskip

\begin{center}
    \begin{tabular}{|c|c|}
    \hline
    region & polynomial $\TL_{T,C}(t)$\\
    \hline
    $V$ & {\tiny $\frac 3 2 t^2+\frac 5 2 t+1$}\\
    \hline
    $E$ & {\tiny $\frac 3 2 t^2 + \frac 3 2 t$}\\
    \hline
    $R_1$ & {\tiny $\frac 3 2 t^2 + \frac 1 2 t$}\\
    \hline
    $R_2$ & {\tiny $\frac 3 2 t^2 - \frac  1 2 t$}\\
    \hline
    \end{tabular}
\end{center}

\begin{figure}
\includegraphics[width=12cm,pagebox=cropbox]{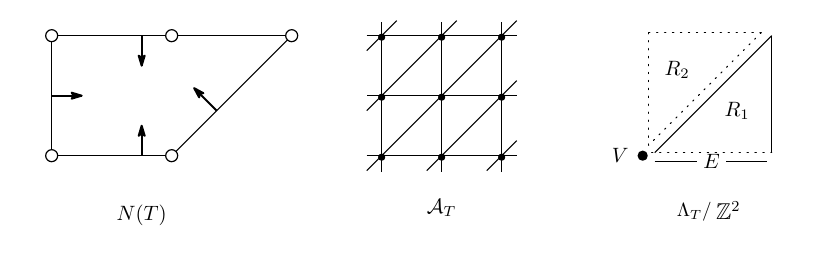}
\caption{$N(T)$, $\mathcal A_T$ and $\Lambda_T/\mathbb Z^2$.}
\label{fig3}
\end{figure}
\end{example}

For a quasi-polynomial $f$,
let $\mathrm{Const}(f)$ be the set of constituents of $f$.
Since the $k$th constituent of $\ehr_{P+\bm v}$ is the $k$th constituent of $\TL_{P,k\bm v}$,
the second statement of the above theorem gives the following finiteness result for constituents of Ehrhart quasi-polynomials of translated polytopes.

\begin{corollary}
\label{finiteness}
If $P \subset \mathbb R^d$ is a rational $d$-polytope, then
\[
\textstyle
\# \big (\bigcup_{\bm v \in \mathbb Q^d} \mathrm{Const}\big(\ehr_{P+\bm v} \big) \big) < \infty.\]
\end{corollary}

\subsection{Polyhedral decompositions associated with hyperplane arrangements}

Theorem \ref{main:tlpregion} is slightly different to Theorem \ref{thm:1-2} in the Introduction (indeed the cell complex $\Lambda_T/\mathbb Z^2$ has $4$ cells while $\Delta_T/\mathbb Z^2$ has 6 cells),
but it can be considered as a refined version of Theorem \ref{thm:1-2}.
We explain this in the rest of this section.

Let $P \subset \mathbb R^d$ be a rational $d$-polytope and $N(P)=\{\bm a_1,\dots,\bm a_m\}$.
The arrangement $\A_P$ gives a natural polyhedral decomposition of $\mathbb R^d$ whose maximal open cells are connected components of $\mathbb R^d \setminus (\bigcup_{H \in \A_P} H)$.
We write $\Delta_P$ for this polyhedral complex.
Note that this $\Delta_P$ is the same as the one defined in the Introduction.
Since any half line $\bm v + \{s \bm w \mid s \in \mathbb R_{\geq 0}\}$,
where $\bm v \in \mathbb R^d$ and $\bm 0 \ne \bm w \in \mathbb R^d$,
must hit one of $H_{\bm a_i,k} \in \A_P$, each connected component of $\mathbb R^d \setminus (\bigcup_{H \in \A_P} H)$ is a bounded set, so $\Delta_P$ is actually a polytopal complex.
By the definition of $\A_P$,
each open cell $A$ of $\Delta_P$ can be written in the form
$$A=A_1 \cap A_2 \cap \cdots \cap A_m$$
such that each $A_i$ is either $H_{\bm a_i,k}$ or $\{\bm x \in \mathbb R^d \mid k< (\bm a_i,\bm x)< k+1\}$.
This means that
each upper region in $\Lambda_P$ can be written as a disjoint union of open cells in $\Delta_P$, in particular,
each element in $\Lambda_P/\mathbb Z^d$ can be written as a disjoint union of elements in $\Delta_P /\mathbb Z^d$.
This proves Theorem \ref{thm:1-2}.

\begin{example}
    Consider the trapezoid $T$ from the Introduction.
    As one can see from Figures \ref{trapezoid} and \ref{fig3},
    $\Lambda_T/\mathbb Z^2$ consists of 4 elements $V,E,R_1,R_2$ and $\Delta_T/\mathbb Z^2$ consists of $6$ elements $V_1,E_1,E_2,E_3,F_1,F_2$.
    We have
    \[
    V=V_1,\ E=E_3,\ R_1=E_1 \cup E_2 \cup F_2,\ R_2=F_1.
    \]
\end{example}

While $\Lambda_P$ and $\Delta_P$ are different in general,
there is a nice case that we have $\Lambda_P=\Delta_P$.
If the set of normal vectors of $P$ is the set of the form $\{\pm \bm a_1,\dots,\pm \bm a_l\}$ then each upper region $R \in \Lambda_P$ must be a region of the form
$$
R=\bigcap_{i=1}^m \big\{\bm x \in \mathbb R^d \mid c_i-1< (\bm a_i,\bm x) \leq c_i\ \mbox{ and } c_i'-1 < (-\bm a_i,\bm x) \leq c_i'\big\}.$$
Each non-empty content in the RHS equals either $H_{\bm a_i,c_i}$ or $\{\bm x \in \mathbb R^d \mid c_i-1< (\bm a_i,\bm x)< c_i\}$
so we have $\Lambda_P=\Delta_P$ in that case.
To summarize, we get the following statement.

\begin{proposition}
If $P \subset \mathbb R^d$ is a $d$-polytope with $N(P)=-N(P)$ then $\Lambda_P=\Delta_P$.
\end{proposition}

A typical example of a polytope $P$ satisfying $N(P)=-N(P)$ is a centrally symmetric polytope $P$ with $P=-P$ (or more generally, a polytope $P$ with $P=-P+ \intvec$ for some $\intvec \in \mathbb Z^d$).

\begin{remark}
Each element of $\Delta_P/\mathbb Z^d$ is an open cell ball, so 
$\Delta_P/\mathbb Z^d$ is indeed a CW complex.
To see that each element of $\Delta_P/\mathbb Z^d$ is a ball, it suffices to check that for each $C \in \Delta_P$ the restriction of $\mathbb R^d \to \mathbb R^d / \mathbb Z^d$ to $C$ is injective.
This injectivity follows from Corollary \ref{stdidealregion} since,
if we have $\bm u=\bm v +\intvec$ for some $\bm u,\bm v \in C$ and $\bm 0 \ne \intvec \in \mathbb Z^d$, then $\mathcal C_P+(\bm v,0)$ and $\mathcal C_P+(\bm u,0) =\big(\mathcal C_P+(\bm v,0)\big)+(\intvec,0)$ must have different sets of integer points.
\end{remark}

\section{Some examples}

Let $P \subset \mathbb R^d$ be a rational $d$-polytope.
We saw in the previous section that $\TL_{P,\bm v}$ only depends on the cell $C$ in $\Delta_P/\mathbb Z^d$ (or the upper region $C$ in $\Lambda_P/ \mathbb Z^d$) with $[\bm v] \in C$,
so for $C \in \Delta_P/\mathbb Z^d$ (or $C \in \Lambda_P/\mathbb Z^d$) we write $\TL_{P,C}=\TL_{P,\bm v}$ with $[\bm v ] \in C$.
In this section, we give a few examples of the computations of $\ehr_{P+\bm v}$ using translated lattice points enumerators.

\begin{example}
    \label{ex5.1}
    Consider the lattice parallelogram $P$ with vertices $(0,0),(1,0),(1,3),(2,3)$.
    Then $N(P)=\{(0,1),(0,-1),(3,-1),(-3,1)\}$ and $\Delta_P/\mathbb Z^2$ $(=\Lambda_P/\mathbb Z^2)$ consists of three vertices
    $V_1,V_2,V_3$, $6$ edges $E_1,E_2,\dots,E_6$ and three $2$-dimensional open cells $R_1,R_2,R_3$ shown in Figure \ref{fig4}.
Since $P$ is a lattice polytope,
translated lattice points enumerators of $P$ are actually polynomials.
The table below is a list of the polynomials $\TL_{P,C}(t)$.

\begin{center}
    \begin{tabular}{|c|c|c|}
\hline
 cell     &  polynomial $\TL_{P,C}(t)$ \\
     \hline
$V_1$     &  {\small $3t^2+2t+1$}  \\
\hline
$V_2,V_3$     & {\small $3t^2+2t$}\\
\hline
$E_1,\dots,E_6$     & {\small $3t^2+t$} \\
\hline
$R_1,R_2,R_3$     &  {\small $3t^2$} \\
\hline
\end{tabular}
\end{center}
\medskip

\begin{figure}
    \centering
    \includegraphics[width=12cm,pagebox=cropbox]{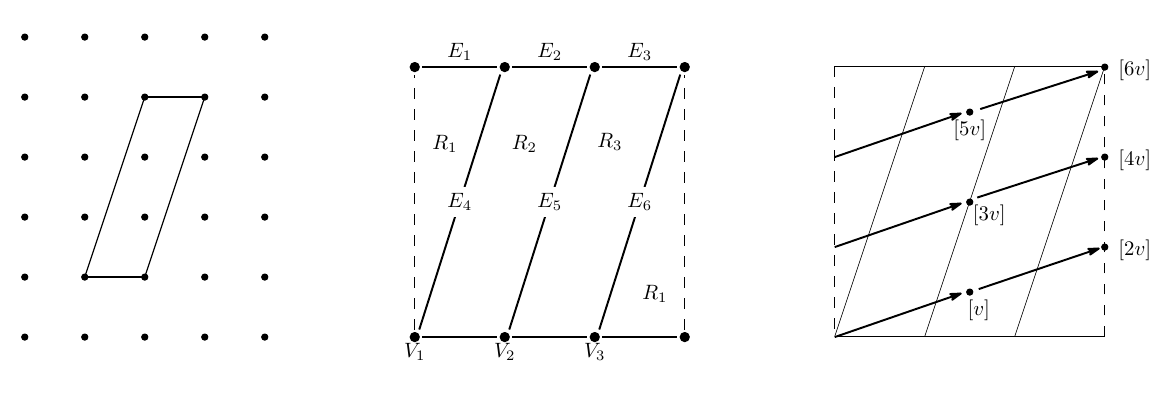}
    \caption{Parallelogram $P$ with vertices $(0,0),(1,0),(1,3),(2,3)$, the cell complex $\Delta_P/\mathbb Z^2$ and positions of $[k\bm v]$ when $\bm v=(\frac 1 3, \frac 1 6)$. All the cells of $\Delta_P/\mathbb Z^2$ in the figure are open cells.}
    \label{fig4}
\end{figure}

Now we compute $\ehr_{P+\bm v}(t)$ when $\bm v=(\frac 1 3, \frac 1 6)$ using this information.
One can compute the constituents of $\ehr_{P+\bm v}$
visually by drawing a line of direction $\bm v$ in $\mathbb R^2/\mathbb Z^2$ and plot the points $[k\bm v]$ for $k=0,1,2,\dots$ as follows.
First, by drawing points $[k\bm v]$ on $\mathbb R^2/\mathbb Z^2$ for $k=0,1,2,\dots$, one can see
\begin{align*}
    [k \bm v] \in
    \begin{cases}
    V_1, & (k \equiv 0 \ (\mathrm{mod}\ 6)),\\
    R_3, & (k \equiv 1 \ (\mathrm{mod}\ 6)),\\
    R_1, & (k \equiv 2,4 \ (\mathrm{mod}\ 6)),\\
    E_5, & (k \equiv 3 \ (\mathrm{mod}\ 6)),\\
    R_2, & (k \equiv 5 \ (\mathrm{mod}\ 6)).
    \end{cases}
\end{align*}
See the second and the third figures in Figure \ref{fig4}.
Lemma \ref{ehrToTrans} says that the $k$th constituent of $\ehr_{P+\bm v}$ is nothing but the $k$th constituent of $\TL_{P,C}$ with $[k\bm v] \in C$.
Hence we get
\begin{align*}
    \ehr_{P+\bm v}(t)=
    \begin{cases}
    \TL_{P,V_1}(t)=3t^2+2t+1, & (t \equiv 0 \ (\mathrm{mod}\ 6)),\\
    \TL_{P,R_3}(t)=3t^2, & (t \equiv 1 \ (\mathrm{mod}\ 6)),\\
    \TL_{P,R_1}(t)=3t^2, & (t \equiv 2,4 \ (\mathrm{mod}\ 6)),\\
    \TL_{P,E_5}(t)=3t^2+t, & (t \equiv 3 \ (\mathrm{mod}\ 6)), \\
    \TL_{P,R_2}(t)=3t^2, & (t \equiv 5 \ (\mathrm{mod}\ 6)).
    \end{cases}
\end{align*}

We remark that parallelogram is a special case of a zonotope,
and a nice combinatorial formula of the Ehrhart quasi-polynomial of a translated integral zonotope is given in \cite[Proposition 3.1]{ABM}.
\end{example}

\begin{example} \label{ex:4.2}
We give a more complicated example.
    Consider the rhombus $Q \subset \mathbb R^2$ having vertices $(\pm 1,0)$ and $(0,\pm \frac 1 2)$.
    Then the cell complex $\Delta_Q/\mathbb Z^2$ consists of four vertices, eight edges and four $2$-dimensional cells. See Figure \ref{fig:rhomb}.
\begin{figure}[h]
    \centering
\includegraphics[width=8cm,pagebox=cropbox]{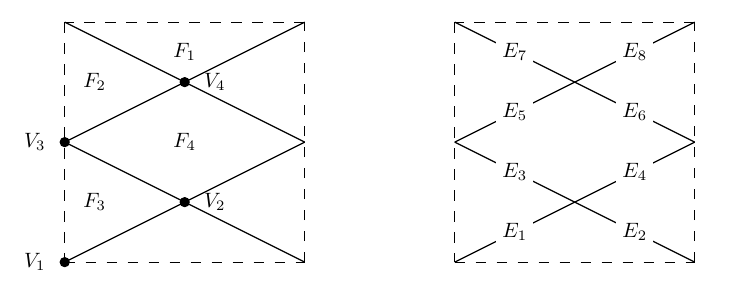}
    \caption{Cell complex associated with $Q$}
    \label{fig:rhomb}
\end{figure}

    Since $2Q$ is integral, $\TL_{Q,C}$ is a quasi-polynomial having period $2$ for each $C \in \Delta_Q/\mathbb Z^2$.
    For a quasi-polynomial $f$ with period $2$,
    we write $f=(f_0,f_1)$, where $f_k$ is the $k$th constituent of $f$. Below is the table of translated lattice points enumerators of $Q$.

    \begin{center}
    \begin{tabular}{|c|c|c|}
\hline
 cell     &  quasi-polynomial $\TL_{Q,C}$ \\
     \hline
$V_1$     &  {\small $(t^2+t+1,t^2+t+1)$}  \\
\hline
$V_2,V_3,V_4$     & {\small $(t^2+t,t^2+t)$}\\
\hline
$E_1,E_2,E_7,E_8$     & {\tiny $(t^2+\frac 1 2 t,t^2+\frac 1 2 t+\frac 1 2)$} \\
\hline
$E_3,E_4,E_5,E_6$     &  {\tiny $(t^2+\frac 1 2 t,t^2+\frac 1 2 t-\frac 1 2 )$} \\
\hline
$F_1$     &  {\small $(t^2,t^2+1)$} \\
\hline
$F_2,F_3$     &  {\small $(t^2,t^2)$} \\
\hline
$F_4$     &  {\small $(t^2,t^2-1)$} \\
\hline
\end{tabular}
\end{center}
\medskip

Consider $\bm u=(\frac 1 8, \frac 1 8)$ and $\bm w =(\frac 1 3, \frac 1 3)$.
Then
\begin{align*}
    [k \bm u] \in
    \begin{cases}
    V_1, & (k \equiv 0 \ (\mathrm{mod}\ 8)),\\
    F_3, & (k \equiv 1,2 \ (\mathrm{mod}\ 8)),\\
    F_4, & (k \equiv 3,4,5 \ (\mathrm{mod}\ 8)),\\
    F_2, & (k \equiv 6,7 \ (\mathrm{mod}\ 8)),
    \end{cases}\ \ \ 
    \mbox{ and } \ \ \ 
    [k \bm w] \in
    \begin{cases}
    V_1, & (k \equiv 0 \ (\mathrm{mod}\ 3)),\\
    E_3, & (k \equiv 1 \ (\mathrm{mod}\ 3)),\\
    E_6, & (k \equiv 2 \ (\mathrm{mod}\ 3)).
    \end{cases}
\end{align*}
Using that the $k$th constituent of $\ehr_{Q+\bm u}$ equals the $k$th constituent of $\TL_{Q,k\bm u}$,
it follows that
\begin{align*}
    \ehr_{Q+\bm u}(t)=
    \begin{cases}
    t^2+t+1, & (t \equiv 0 \ (\mathrm{mod}\ 8)),\\
    t^2, & (t \equiv 1,2,4,6,7 \ (\mathrm{mod}\ 8)),\\
    t^2-1, & (t \equiv 3,5 \ (\mathrm{mod}\ 8)),
    \end{cases}
\end{align*}
and
\begin{align*}
    \ehr_{Q+\bm w}(t)=
    \begin{cases}
    t^2+t+1, & (t \equiv 0,3 \ (\mathrm{mod}\ 6)),\\
    t^2+\frac 1 2 t-\frac 1 2, & (t \equiv 1,5 \ (\mathrm{mod}\ 6)),\\
    t^2+\frac 1 2 t, & (t \equiv 2,4 \ (\mathrm{mod}\ 6)).
    \end{cases}
\end{align*}
One can also see from the second example that the minimum period of $\ehr_{Q+\bm w}$ is not necessary the denominator of $\bm w$.
\end{example}

\section{Reciprocity in maximal regions}
In this section, we explain that the quasi-polynomials $\TL_{P,C}$ for maximal dimensional cells $C \in \Delta_P/\mathbb Z^d$ have a reciprocity which comes from the reciprocity in Theorem \ref{thm:McMullen}.

\subsection{Reciprocity}
Let $P \subset \mathbb R^d$ be a rational $d$-polytope.
The reciprocity in Theorem \ref{thm:McMullen} says that, for any $\bm v \in \mathbb R^d$, one has
\begin{align}
\label{555}    
\TL_{\mathrm{int}(P),\bm v}(t)=(-1)^d \TL_{P,-\bm v}(-t)\ \ \ \mbox{ for } t \in \mathbb Z_{>0}.
\end{align}
Since the affine hyperplane arrangement $\A_P$ is centrally symmetric, that is $-\A_P=\A_P$,
we have $R \in \Delta_P$ if and only if $-R \in \Delta_P$.
For each $C \in \Delta_P/\mathbb Z^d$ with a representative $R \in \Delta_P$,
we write $-C$ for the element of $\Delta_P/\mathbb Z^d$ corresponding to the cell $-R$.
By Theorem \ref{thm:1-2} and \eqref{555}
we have $\TL_{\mathrm{int}(P),\bm v}(t)=\TL_{\mathrm{int}(P),\bm u}(t)$ when $[\bm u]$ and $[\bm v]$ belong to the same cell in $\Delta_P/\mathbb Z^d$.
Thus, for each $C \in \Delta_P/\mathbb Z^d$,
we write $\TL_{\mathrm{int}(P),C}(t)=\TL_{\mathrm{int}(P),\bm v}(t)$ with $[\bm v] \in C$.
Using this notation, \eqref{555} can be written in the following form.

\begin{proposition}
\label{prop5.1}
Let $P \subset \mathbb R^d$ be a rational $d$-polytope.
For any $C\in \Delta_P/\mathbb Z^d$, one has
\[
\TL_{\mathrm{int}(P),C}(t)=
(-1)^d \TL_{P,-C}(-t)
\ \ \ \mbox {for } t \in \mathbb Z_{>0}.
\]
\end{proposition}

Note that the above equation says that
$\TL_{\mathrm{int}(P),C}$ is a quasi-polynomial on $\mathbb Z_{>0}$.

\begin{example}
Consider the trapezoid $T$ in the Introduction.
We have
\[
F_1=-F_2,\ E_1=-E_1,\ E_2=-E_2, E_3=-E_3, V_1=-V_1.
\]
The following tables are lists of polynomials $\TL_{T,C}(t)$ and $\TL_{\mathrm{int}(T),C}(t)$.
\medskip

    \centering
\begin{tabular}{|c|c|c|}
\hline
cell     &  polynomial $\TL_{T,C}(t)$ \\
     \hline
$V_1$     &  {\tiny $\frac 3 2 t^2+ \frac 3 2 t+1$}  \\
\hline
$E_1,E_2$     &  {\tiny $\frac 3 2 t^2+\frac 1 2 t$} \\
\hline
$E_3$     &  {\tiny $\frac 3 2 t^2+\frac 3 2 t$} \\
\hline
$F_1$     &  {\tiny $\frac 3 2 t^2-\frac 1 2 t$} \\
\hline
$F_2$     &  {\tiny $\frac 3 2 t^2+\frac 1 2 t$} \\
\hline
\end{tabular}
\ \ \ \begin{tabular}{|c|c|c|}
\hline
 cell     &  polynomial $\TL_{\mathrm{int}(T),C}(t)$ \\
     \hline
$V_1$     &  {\tiny $\frac 3 2 t^2- \frac 3 2 t+1$}  \\
\hline
$E_1,E_2$     & {\tiny $\frac 3 2 t^2-\frac 1 2 t$ }\\
\hline
$E_3$     & {\tiny $\frac 3 2 t^2-\frac 3 2 t$} \\
\hline
$F_2$     &  {\tiny $\frac 3 2 t^2+\frac 1 2 t$} \\
\hline
$F_1$     &  {\tiny $\frac 3 2 t^2-\frac 1 2 t$} \\
\hline
\end{tabular}
\end{example}

\subsection{Maximal cells}

Let $P \subset \mathbb R^d$ be a rational $d$-polytope with the unique irredundant presentation
\begin{align}
    \label{6-1} P= H_{\bm a_1,b_1}^{\geq} \cap \cdots \cap H_{\bm a_m,b_m}^\geq.
\end{align}
We write $F_i=P \cap H_{\bm a_i,b_i}$ for the facet of $P$ which lies in the hyperplane $H_{\bm a_i,b_i}$.
The next lemma follows from Lemmas \ref{lem:gcd} and \ref{lem:cone}.

\begin{lemma}
\label{lem:6.1}
With the same notation as above,
for $\bm v \in \mathbb R^d$,
the cone $\mathcal C_{F_i} +(\bm v,0)$ contains a lattice point if and only if $\bm v \in H_{\bm a_i,k}$ for some $k \in \mathbb Z$.
\end{lemma}

The lemma says that the cone $\mathcal C_P+(\bm v,0)$ has no lattice points in its boundary if $\bm v \in \mathbb R^d \setminus \bigcup_{H \in \A_P} H$,
equivalently, if $[\bm v]$ belongs to a $d$-dimensional cell of $\Delta_P/\mathbb Z^d$.
Hence we have

\begin{corollary}
\label{cor:TLrecipro}
If $P \subset \mathbb R^d$ is a rational $d$-polytope
and $C$ is a $d$-dimensional cell of $\Delta_P /\mathbb Z^d$, then
\[ \TL_{P,C}(t)=(-1)^d\TL_{P,-C}(-t) \ \ \mbox{ for all } t \in \mathbb Z_{\geq 0}.
\]
\end{corollary}

\begin{proof}
    Let $\bm v \in \mathbb R^d$ such that $[\bm v] \in C$.
    Then $\bm v \not \in H$ for any $H \in \A_P$, which implies that the cone $\mathcal C_P+(\bm v,0)$ has no lattice points in its boundary.
    Thus by Proposition \ref{prop5.1} we have
    \[
    \TL_{P,C}(t)=\TL_{\mathrm{int}(P),C}(t)=(-1)^d\TL_{P,-C}(-t) \ \ \ \mbox{ for }t \in \mathbb Z_{>0}.
    \]
    Since $\TL_{P,C}$ and $\TL_{P,-C}$ are quasi-polynomials on $\mathbb Z_{\geq 0}$,
    this implies the desired equality.
\end{proof}

The above reciprocity has a special meaning for centrally symmetric polytopes.
Looking at the quasi-polynomials $\TL_{Q,F_i}$ in Example \ref{ex:4.2}, one may notice that each constituent is a polynomial in $t^2$. In other words, the linear term $t$ vanishes. 
We explain that this has a reason.
We first remind the following easy fact.

\begin{lemma} \label{lem:cs}
    Let $P \subset \mathbb R^d$ be a rational polytope with $-P=P+\intvec$ for some $\intvec \in \mathbb Z^d$. Then $\TL_{P,\bm v}(t)=\TL_{P,-\bm v}(t)$ for any $\bm v \in \mathbb R^d$ and $t \in \mathbb Z_{\geq 0}$.
\end{lemma}

\begin{proof}
The assertion follows since, for each integer $k \geq 0$, the correspondence $\bm x \to -\bm x -k \intvec$
give a bijection between lattice points in $kP+\bm v$
and those in $-kP-\bm v-k\intvec =kP-\bm v$.
\end{proof}

\begin{theorem}
\label{thm:cssymmetry}
Let $P \subset \mathbb R^d$ be a rational $d$-polytope with $-P=P+ \intvec$ for some $\intvec \in \mathbb Z^d$
and let $C \in \Delta_P /\mathbb Z^d$ be a $d$-dimensional cell. 
Let $f(t)$ be the $k$th constituent of $\TL_{P,C}$
and let $g(t)$ be the $(-k)$th constituent of $\TL_{P,C}$. Then
$$f(t)=(-1)^d g(-t).$$
\end{theorem}

\begin{proof}
    Corollary \ref{cor:TLrecipro} and Lemma \ref{lem:cs} say
    \[
    \TL_{P,C}(t)=(-1)^d \TL_{P,-C}(-t)=(-1)^d \TL_{P,C}(-t) \ \ \mbox{ for all }t \in \mathbb Z_{\geq 0}.
    \]
By considering the $k$th constituent in the above equality, we get the desired assertion.
\end{proof}

If a polynomial $f(t)$ of degree $d$ satisfies
$f(t)=(-1)^d f(-t)$,
then it must be a polynomial in $t^2$ when $d$ is even 
and $t$ times a polynomial in $t^2$ when $d$ is odd.
Hence we get the following corollary,
which explains a reason why we get polynomials in $t^2$ in Example \ref{ex:4.2}.

\begin{corollary}
    With the same notation as in Theorem \ref{thm:cssymmetry},    
\begin{itemize}
    \item [(1)] the $0$th constituent of $\TL_{P,C}(t)$ is either a polynomial in $\mathbb Q[t^2]$ or $t \mathbb Q[t^2]$;
    \item[(2)] if $2P$ is integral, then the $1$st constituent of $\TL_{P,C}(t)$ is either a polynomial in $\mathbb Q[t^2]$ or $t \mathbb Q[t^2]$.
\end{itemize}
\end{corollary}

Note that when $2P$ is integral the quasi-polynomial $\TL_{P,C}$ has period 2, so its $1$st constituent equals its $(-1)$th constituent.

\section{Translated lattice points enumerators determine polytopes}

It is clear that if $P =Q+\intvec$ for some integer vector $\intvec$,
then $\TL_{P,\bm v}=\TL_{Q,\bm v}$ for all  vectors $\bm v$.
The goal of this section is to prove the converse of this simple fact,
which is equivalent to Theorem \ref{thm:1-4} in the Introduction by Lemma \ref{ehrToTrans}\footnote{The condition ``$\TL_{P,\bm v}=\TL_{Q,\bm v}$ for all $\bm v \in \mathbb Q^d$" is equivalent to the condition ``$\TL_{P,\bm v}=\TL_{Q,\bm v}$ for all $\bm v \in \mathbb R^d$".}.

\begin{theorem}
\label{thm:6-1}
    Let $P$ and $Q$ be rational $d$-polytopes in $\mathbb R^d$. If $\TL_{P,\bm v}(t)=\TL_{Q,\bm v}(t)$ for all $\bm v \in \mathbb R^d$ and $t \in \mathbb Z_{\geq 0}$ then $P=Q+ \intvec$ for some $\intvec \in \mathbb Z^d$.
\end{theorem}

To simplify notation, we use the notation
$$\Gamma_P=\big\{\big(\bm v,\TL_{P,\bm v}(t)\big) \in \mathbb R^d \times \mathcal {QP}\mid \bm v\in \mathbb R^d\big\},$$
where $\mathcal {QP}$ is the set of all quasi-polynomials in $t$.
Thus,
 what we want to prove is that $\Gamma_P=\Gamma_Q$ implies $P=Q+\intvec$ for some $\intvec \in \mathbb Z^d$.

To prove the theorem,
we first recall Minkowski's theorem,
which says that normal vectors and volumes of facets determine a polytope.
Let $P \subset \mathbb R^d$ be a $d$-polytope with irredundant presentation $P=\bigcap_
{i=1}^m H_{\bm a_i,b_i}^\geq$,
where $\|\bm a_i\|=1$,
and let $F_i=P \cap H_{\bm a_i,b_i}$
be the facet of $P$ which lies in the hyperplane $H_{\bm a_i,b_i}$.
We write
$$\mathcal M(P)=\big\{\big(\bm a_1,\mathrm{vol}(F_1)\big),\dots,(\bm a_m,\mathrm{vol}(F_m))\big\},$$
where $\mathrm{vol}(F_i)$ is the relative volume of $F_i$.
The following result is known as Minkowski's theorem (see \cite[\S6.3 Theorem 1]{Alexandrov}).

\begin{theorem}[Minkowski]
\label{thm:Minkowski}
If $P$ and $Q$ are $d$-polytopes in $\mathbb R^d$ with $\mathcal M(P)=\mathcal M(Q)$, then $P=Q+\bm v$ for some $\bm v \in \mathbb R^d$.
\end{theorem}

To apply Minkowski's theorem in our situation,
we will show that we can know volumes of facets of a polytope from translated lattice points enumerator on codimension $1$ cells of $\Delta_P/\mathbb Z^d$.
We say that a point $\bm x \in H_{\bm a,k} \in \A_P$ is {\bf generic} in $\A_P$ if $\bm x \not \in H$ for any $H \in \A_P$ with $H \ne H_{\bm a,k}$.
Note that $\bm x \in H_{\bm a,k}$ is generic if and only if it is contained in a $(d-1)$-dimensional cell of $\Delta_P$.

\begin{lemma}
    \label{lem:codim1}
    Let $P \subset \mathbb R^d$ be a rational $d$-polytope,
    $\bm a \in N(P)$,
    and let $F$ be a facet of $P$ corresponding to the normal vector $\bm a$.
    If $\bm v \in H_{\bm a,k}$ is generic in $\A_P$, then for all sufficiently small $\varepsilon>0$, one has
    \begin{itemize}
        \item [(1)]
        $\TL_{P,\bm v}-\TL_{P,\bm v+\varepsilon \bm a}=\TL_{F,\bm v} \ne 0$.
        \item[(2)] $\TL_{P,\bm v}-\TL_{P,\bm v-\varepsilon \bm a}=0$ if there is no $c \in \mathbb R_{>0}$ such that $- c\bm a \in N(P)$. 
    \end{itemize}
\end{lemma}

\begin{proof}
The fact that $\TL_{F,\bm v} \ne 0$ follows from Lemma \ref{lem:6.1}.
By Lemma \ref{gcdcor}, there is an $\varepsilon>0$
such that
\begin{align*}
\partial \big(\scC_P+(\bm v+s \bm a,0)\big) \cap \mathbb Z^{d+1}
=
\partial \big(\scC_P+(\bm v-s \bm a,0)\big) \cap \mathbb Z^{d+1}=\varnothing
\ \ \ \mbox{for all }0 < s \leq \varepsilon.
\end{align*}
Let $\scC_\varepsilon^+=\scC_P+(\bm v+\varepsilon \bm a,0)$
and $\scC_\varepsilon^-=\scC_P+(\bm v-\varepsilon \bm a,0)$.
By the above equation, we have
\begin{itemize}
    \item[(i)] $(\scC_\varepsilon^+ \cup  \scC_\varepsilon^-) \cap \mathbb Z^{d+1} \subset ( \scC_P + (\bm v,0)\big)\cap \mathbb Z^{d+1}$;
    \item[(ii)] $\mathrm{int}
    \big(\scC_P+(\bm v,0)\big)\cap  \mathbb Z^{d+1} \subset \big( \scC_\varepsilon^+ \cap \scC_\varepsilon^-\big)\cap \mathbb Z^{d+1}$.  
\end{itemize}
Also, regarding lattice points in the boundary of $\scC_P+(\bm v,0)$, we have 
\begin{itemize}
    \item[(iii)] if $\bm x$ is a lattice point in a facet 
$(\scC_P \cap H_{(\bm b,k),0}) +(\bm v,0)$ of $\scC_P+(\bm v,0)$ with $H_{(\bm b,k),0}^{\geq } \supset \scC_P$, then
\begin{align}
\label{6.3-2}
\bm x \in \scC_\varepsilon^{+} \ \ \Leftrightarrow
(\bm b, \bm a) < 0
\ \ \ \mbox{ and }\ \ \ 
\bm x \in \scC_\varepsilon^{-} \ \ \Leftrightarrow
(\bm b, \- \bm a) > 0.
\end{align}
In particular, since $\scC_\varepsilon^+$ and $\scC_\varepsilon^-$ have no lattice points in their boundaries,
a lattice point in the boundary of $\scC_P+(\bm v,0)$ is contained in exactly one of $\scC_\varepsilon^+$ and $\scC_\varepsilon^-$.\end{itemize}

Now we assume that $F$ is the only facet of $P$ that is orthogonal to $\bm a$ and prove (1) and (2).
Observe $\scC_F=\scC_P \cap H_{(\bm a,k),0}$ for some $k \in \mathbb R$.
By the assumption and Lemma \ref{lem:6.1},
$\mathcal \scC_{F}+(\bm v,0)$ is the only facet of $\mathcal \scC_P+(\bm v,0)$ that contains lattice points, so
\begin{align}
\label{6.3-3}   
\partial \big(\mathcal \scC_P +(\bm v,0) \big) \cap \mathbb Z^{d+1}
=
\big(\mathcal \scC_F+(\bm v,0) \big) \big) \cap \mathbb Z^{d+1}.
\end{align}
On the other hand, lattice points in $\scC_P+(\bm v,0)$ are not contained in $\scC_\varepsilon^+$ by (iii),
so by (i) and (ii) we have
\begin{align}
\label{6.3-4}
\big(\scC_P+(\bm v+\varepsilon \bm a,0)\big) \cap \mathbb Z^{d+1}
=\scC_\varepsilon^+\cap \mathbb Z^{d+1}= \mathrm{int}\big( \scC_P+(\bm v,0) \big) \cap \mathbb Z^{d+1}.
\end{align}
Then the equations \eqref{6.3-3} and \eqref{6.3-4} prove (1).
Similarly, 
all lattice points in $\scC_P+(\bm v,0)$ are contained in $\scC_\varepsilon^-$ by (iii),
so again by (i) and (ii) we have
\[
\big(\scC_P+(\bm v-\varepsilon \bm a,0)\big) \cap \mathbb Z^{d+1}
=\scC_\varepsilon^-\cap \mathbb Z^{d+1}
=
\big( \scC_P+(\bm v,0) \big) \cap \mathbb Z^{d+1},
\]
proving (2).

Second, we assume that there is a facet $G \ne F$ of $P$ that is orthogonal to $\bm a$.
This condition is equivalent to the condition that there is $c \in \mathbb R_{>0}$ such that $-c \bm a \in N(P)$.
Also,
the normal vector corresponding to the facet $G$ must be equal to $-c\bm a$
and by the assumption and Lemma \ref{lem:6.1} we have
\begin{align}
\label{6.3-5}    
\partial \big(\mathcal C_P +(\bm v,0) \big) \cap \mathbb Z^{d+1}
=
\big(\big(\mathcal C_F+(\bm v,0) \big) \cap \mathbb Z^{d+1} \big)
\cup \big(\big(\mathcal C_G+(\bm v,0) \big)\big) \cap \mathbb Z^{d+1}\big).
\end{align}
The property (iii) says
\[
\big(\big(\scC_F+(\bm v,0) \big)\cap
\mathbb Z^{d+1} \big)\cap \scC_\varepsilon^+=\varnothing
\mbox{ and }
\big(\scC_G+(\bm v,0) \big)\cap
\mathbb Z^{d+1} \subset \scC_\varepsilon^+.\]
Then
by (i), (ii) and \eqref{6.3-5},
we have
\begin{align*}
\big(\mathcal C_P +(\bm v+ \varepsilon \bm a,0)\big) \cap \mathbb Z^{d+1}
&= \big( \mathrm{int}\big(\mathcal C_P+(\bm v,0)\big) \cap \mathbb Z^{d+1}\big)\cup
\big(\big(\mathcal C_G +(\bm v,0)\big)\cap \mathbb Z^{d+1}\big)\\
&= \big( \big(\mathcal C_P+(\bm v,0)\big) \cap \mathbb Z^{d+1}\big)\setminus
\big(\big(\mathcal C_F +(\bm v,0)\big)\cap \mathbb Z^{d+1}\big)
\end{align*}
proving (1).
\end{proof}

\begin{lemma}
    \label{7.3}
If $P$ and $Q$ are rational $d$-polytopes in $\mathbb R^d$
with $\Gamma_P=\Gamma_Q$,
then $\mathcal M(P)=\mathcal M(Q)$.
\end{lemma}

\begin{proof}
What we must prove is that the set $\Gamma_P$ determines the directions of inner normal vectors of $P$ as well as volumes of the facets of $P$.

By Theorem \ref{main:tlpregion}(1) and Lemma \ref{lem:codim1}(1), $\bm x \in \mathbb R^d\setminus \bigcup_{H \in \A_P} H$ if and only if there is an open ball $B \ni \bm x$ such that $\TL_{P,\bm x}=\TL_{P,\bm y}$ for all $\bm y \in B$.
This says that the set $\Gamma_P$ determines $\A_P$,
and the definition of $\A_P$ says that $\A_P$ determines the set $\overline N=\{\pm (\bm a/\|\bm a\|) \mid \bm a \in N(P)\}$.
For each $\bm a \in \overline N$, Lemma \ref{lem:codim1} also says $c \bm a \in N(P)$ for some 
$c>0$ if and only if, for a generic $\bm x \in H_{\bm a,0} \in \A_P$,
we have $\TL_{P,\bm x} \ne \TL_{P,\bm x+ \varepsilon \bm a}$ for a sufficiently small $\varepsilon>0$.
Hence the set $\Gamma_P$ determines $\{(\bm a/\|\bm a\|) \mid \bm a \in N(P)\}$.

It remains to prove that $\Gamma_P$ determines the volumes of facets of $P$.
Let $F$ be a facet of $P$ and let $\bm a \in N(P)$
be the normal vector associated with the facet $F$.
For any $\bm v \in \mathbb R^d$,
let $\TL_{P,\bm v}^{0}(t)$ denote the $0$th constituent of $\TL_{P,\bm v}$,
which must be a degree $d$ polynomial whose leading coefficient is the normalized volume of $P$.
If we take a generic point $\bm x \in H_{\bm a,0}$ in $\A_P$, then by Lemma \ref{lem:codim1} we have
\[
\textstyle
\mathrm{lim}_{t \to \infty} \frac 1 {t^{d-1}} \TL^{0}_{F,\bm x}(t)
= \mathrm{lim}_{t \to \infty} \frac 1 {t^{d-1}} \big(\TL^{0}_{P,\bm x}(t)-\TL^{0}_{P,\bm x+\varepsilon \bm a}(t) \big),
\]
where $\varepsilon>0$ is sufficiently small.
Since $\TL_{P,\bm x}^{0}(t)-\TL^{0}_{P,\bm x+\varepsilon \bm a}(t)$ is a polynomial of degree $\leq d-1$, this limit exists and must be equal to the relative volume of $F$ since $\TL_{F,\bm x}$ can be considered as a translated lattice points enumerator in the Euclidean space $H_{\bm a,0} \cong \mathbb R^{d-1}$ with the lattice $H_{\bm a,0}\cap \mathbb Z^d \cong \mathbb Z^{d-1}$.
Thus volumes of facets of $P$ are determined by $\Gamma_P$.
\end{proof}

\begin{remark}
If one know $\Gamma_P$ then we can know the volume of $P$ since it appears in the leading coefficient of a constituent of $\mathrm{TL}_{P,\bm v}$.
There is another way to compute the volume of $P$ that was considered in \cite{alh}.
Let $P\subset \mathbb R^d$ be a rational convex polytope.
For each cell $C \in \Delta_P/\mathbb Z^d$,
we call the number $\mathrm{TL}_{P,-C}(1)$
the multiplicity of $C$.
This number is indeed the multiplicity in the sense that, if $\rho :\mathbb R^d \to \mathbb R^d/\mathbb Z^d$ is the natural projection, then
for $[\bm v] \in \mathbb R^d/\mathbb Z^d$ one has
\[
\mathrm{TL}_{P,-\bm v}(1)=\#((P-\bm v) \cap \mathbb Z^d)=
\#(P \cap (\bm v+\mathbb Z^d))=
\# (\rho^{-1}([\bm v])).
\]
This equation says that the volume of $P$ equals to the sum of volumes of (maximal dimensional) cells of $\Delta_P/\mathbb Z^d$ times their multiplicities.
Volumes of facets of $P$ can be also computed using similar argument given in the proof of Lemma \ref{lem:codim1}.
\end{remark}

Let $\pi_i: \mathbb R^d \to \mathbb R^{d-1}$ be the projection given by
$$\pi_i(x_1,\dots,x_d)=(x_1,\dots,x_{i-1},x_{i+1},\dots,x_d).$$
We next show that translated lattice points enumerators of $\pi_i(P)$ can be determined from those of $P$.
Let $P \subset \mathbb R^d$ be a $d$-polytope.
We define
$$\partial_i^-P=\{ \bm x \in P \mid \bm x \not \in (P + \varepsilon \mathbf {e}_i) \mbox{ for all } \varepsilon >0\},$$
where $\mathbf e_1,\dots,\mathbf e_d$ are the standard vectors of $\mathbb R^d$.
Intuitively, $\partial_i^-P$ is the set of points in $P$ which is visible from $-\infty \mathbf e_i$ (see Figure \ref{fig6}).
Indeed, $\partial_i^- P$ has the following description:
Let $\mathrm{Facets}(P)$ be the set of facets of $P$
and assume $P=\bigcup_{F \in \mathrm{Facets}(P)} H^\geq _{\bm a_F,b_F}$.
Then, for any $\bm x \in P$ and $\varepsilon \in \mathbb R$,
we have $\bm x \not \in P+\varepsilon \mathbf e_i$
if and only if $(\bm a_F,\bm x)-\varepsilon (\bm a_F,\mathbf e_i)= (\bm a_F,\bm x-\varepsilon \mathbf e_i) < b_F$ for some $F \in \mathrm{Facets}(P)$.
This means
\begin{align}    \label{7-1}
    \partial_i^- P= \bigcup_{ F \in \mathrm{Facets}(P),\ (\bm a_F,\mathbf e_i)>0} F,
\end{align}
and the RHS of the above equation is nothing but the set of points in $P$ which is visible from $-\infty \mathbf e_i$ (see \cite[\S 5.2]{Gru:book} for more information on visible faces of a polytope).

\begin{figure}[h]
    \centering
    \includegraphics{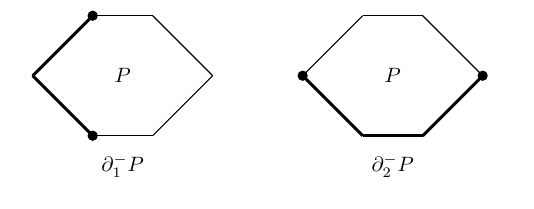}
    \caption{Visualizations of $\partial_1^- P$ and $\partial_2^- P$ when $P$ is the hexagon with vertices $\pm (2,0),\pm (1,1),\pm (1,-1)$. Thick lines correspond to $\partial_1^- P$ and $\partial_2^- P$.}
    \label{fig6}
\end{figure}

\begin{lemma}
    \label{enumProjection}
    With the same notation as above,
    for any $\bm v \in \mathbb R^d$, there is an $\varepsilon_{i,\bm v}>0$ such that
    $$
    (tP+\bm v) \cap \mathbb Z^d=
    \left(\big( t(\partial_i^- P)+\bm v \big) \bigsqcup \big( t P + \bm v + \varepsilon_{i,\bm v} \mathbf e_i \big) \right)\cap \mathbb Z^d
    \ \ \mbox{ for all $t  \in \mathbb Z_{\geq 0}.$}$$
\end{lemma}

We note that when $t=0$, we consider that $t(\partial^-_i P)=\{0\}$ in Lemma \ref{enumProjection}.

\begin{proof}
By Lemma \ref{gcdcor} there is an $\varepsilon>0$ such that
{\small
\begin{align*}
&\big( \mathcal C_P +(\bm v,0 ) \big) \cap \mathbb Z^{d+1}\\
&=\left(\big( \mathcal C_P\!+\! (\bm v \! +\! \varepsilon \mathbf e_i
,0) \big)\bigsqcup
\big\{\bm x\! +\!(\bm v,0) \in \mathcal C_P \!+ \!(\bm v,0) \mid
\bm x \not \in \mathcal C_P \! +\! s(\mathbf e_i,0) \
\mbox{ for all }s>0\big\}
\right) \cap \mathbb Z^{d+1}.
\end{align*}
}
\hspace{-6pt}
Cutting the above equation by the hyperplane $x_{d+1}=t$,
we get the desired equality.
\end{proof}

We define $\TL_{P,\bm v}^{(-i)}(t)$ by
\[\TL_{P,\bm v}^{(-i)}(t)=
\# \big(\big( t(\partial_i^- P)+ \bm v \big) \cap \mathbb Z^{d}\big).\]
Lemma \ref{enumProjection} says that
\[\TL_{P,\bm v}^{(-i)}(t)
=\TL_{P,\bm v}(t)-\TL_{P,\bm v+\varepsilon_{i,\bm v} \mathbf e_i}(t),\]
where $\varepsilon_{i,\bm v}$ is a number given in Lemma \ref{enumProjection}.
We note that the function $\TL_{P,\bm v}^{(-i)}$ is zero for almost all $\bm v \in \mathbb R^d$.
Indeed,
we have the following statement.

\begin{lemma}
    \label{7-5}
    With the same notation as above,
    $\TL_{P,\bm v}^{(-i)}$ is not a zero function only when there is $\bm a \in N(P)$ and $k \in \mathbb Z$ such that $\bm v \in H_{\bm a,k}$ and $(\bm a,\mathbf e_i)>0$.
\end{lemma}

\begin{proof}
We have $\TL_{P,\bm v}^{(-i)} \ne 0$ only when $$\big( \mathcal C_{\partial_i^-P} + (\bm v,0) \big) \cap \mathbb Z^{d+1} \ne  \varnothing.$$
    By \eqref{7-1} and Lemma \ref{lem:6.1} this condition is equivalent to $\bm v  \in H_{\bm a,k}$ for some $\bm a \in N(P)$ and $k \in \mathbb Z$ with $(\bm a,\mathbf e_i)>0$.
\end{proof}

The next proposition shows that translated lattice points enumerators of $\pi_i(P)$ can be determined from those of $P$.

\begin{proposition}
\label{7.6}
    Let $P \subset \mathbb R^d$ be a rational $d$-polytope. For any $\bm v \in \mathbb R^d$ and $t \in \mathbb Z_{\geq 0}$, one has
    $$\TL_{\pi_i(P),\pi_i(\bm v)}(t)= \sum_{0 \leq s <1,\ \TL_{P,\bm v + s \mathbf e_i}^{(-i)}(t) \ne 0} \TL_{P,\bm v+ s \mathbf e_i}^{(-i)}(t).$$
\end{proposition}

We note that the RHS in the proposition is a finite sum by Lemma \ref{7-5} since the segment $\{\bm v+s \mathbf e_i \mid 0 \leq s<1\}$ meets only finitely many hyperplanes in $\A_P$.
See Figure \ref{fig:projmove} for a visualization of the proposition.
\begin{figure}
    \centering
\includegraphics[width=6cm,pagebox=cropbox]{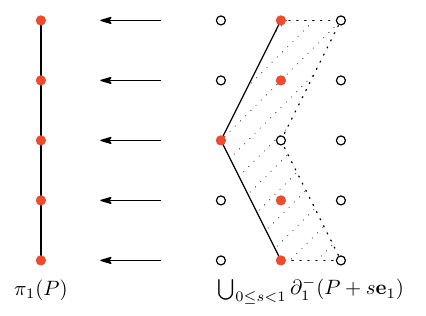}
    \caption{Lattice points in the projection.}
    \label{fig:projmove}
\end{figure}

\begin{proof}
    We may assume $i=d.$ Fix $t \in \mathbb Z_{\geq 0}$ and a lattice point $\intvec=(n_1,\dots,n_{d-1}) \in \pi_d(tP+\bm v)$.
    It suffices to prove that there is a unique integer $r \in \mathbb Z$ such that $(\intvec,r) \in \bigcup_{0 \leq s <1}(t(\partial_d^-P)+\bm v+s \mathbf e_d)$.

    (Existence)\
    By the assumption, there is an $\alpha \in \mathbb R$ such that
    $$(\intvec,\alpha) \in  t (\partial_d^-P)+\bm v.$$
    Then, $r=\lceil \alpha \rceil$ satisfies the desired condition since $(\intvec,\lceil \alpha \rceil)$ is contained in $t(\partial_d^-P)+\bm v+ (\lceil \alpha \rceil -\alpha)\mathbf e_d$.

    (Uniqueness)\
    The uniqueness of $r$ follows from the fact that,
    for any $(\intvec,\alpha),(\intvec,\alpha')$ which are contained in $\bigcup_{0 \leq s <1} (t(\partial_n^-P)+\bm v+s \mathbf e_d)$, we have $|\alpha-\alpha'|<1$.
\end{proof}

We will also use the following variation of Proposition \ref{7.6}.
For a $d$-polytope $P$,
let
\[
\partial_i^+P=\{ \bm x \in P \mid \bm x \not \in P - \varepsilon \mathbf e_i \mbox{ for all } \varepsilon>0\}
\]
and
\[
\TL_{P,\bm v}^{(+i)}(t)=\# \big( \big( t (\partial_i^+ P)+ \bm v \big) \cap \mathbb Z^d \big)
\ \ \mbox{ for }t \in \mathbb Z_{\geq 0}.
\]
The next statement can be proved by the same argument given in the proof of Proposition \ref{7.6}.

\begin{proposition}
\label{7.6*}
    Let $P \subset \mathbb R^d$ be a rational $d$-polytope. For any $\bm v \in \mathbb R^d$ and $t \in \mathbb Z_{\geq 0}$, one has
    $$\TL_{\pi_i(P),\pi_i(\bm v)}(t)= \sum_{0 \leq s <1,\ \TL_{P,\bm v - s \mathbf e_i}^{(+i)}(t) \ne 0} \TL_{P,\bm v- s \mathbf e_i}^{(+i)}(t).$$
\end{proposition}

We now prove Theorem \ref{thm:6-1}.

\begin{proof}[Proof of Theorem \ref{thm:6-1}]
We use induction on $d$.
Suppose $d=1$ and $\Gamma_P=\Gamma_Q$.
Let $\ell=\#(P \cap \mathbb Z)$ and let $p$ be the maximal integer which is equal to or smaller than $\min P$.
Then, by setting
\[
a=\min \{s \in [0,1)\mid \#((P+s)\cap \mathbb Z)-\#((P+s+\varepsilon)\cap \mathbb Z) \ne 0\}
\]
and
\[
b=\min \{s \in [0,1) \mid \#((P-s)\cap \mathbb Z)-\#(P-s-\varepsilon)\cap \mathbb Z \ne 0\},
\]
where $\varepsilon>0$ is sufficiently small.  
we have 
\[P=[p+(1-a),p+\ell+b]\]
Since $\ell,a,b$ only depend on $\Gamma_P$,
this implies $P=Q+n$ for some $n \in \mathbb Z$ .


Assume $d>1$ and $\Gamma_P=\Gamma_Q$.
By Lemma \ref{7.3}, we already know $Q=P+\bm v$ for some $\bm v \in \mathbb R^d$.
By Proposition \ref{7.6} and the assumption $\Gamma_P=\Gamma_Q$, we have $\Gamma_{\pi_1(P)}=\Gamma_{\pi_1(Q)}$ and $\Gamma_{\pi_2(P)}=\Gamma_{\pi_2(Q)}$.
Since $Q=P+\bm v$, the induction hypothesis says that $\pi_1(\bm v),\pi_2(\bm v) \in \mathbb Z^{d-1}$ which guarantees $\bm v \in \mathbb Z^d$.
\end{proof}

\section{Group symmetry}

In the previous section, we saw that the translated lattice points enumerators determine polytopes up to translations by integer vectors.
In this section,
we study translated lattice points enumerators of polytopes having some symmetries,
in particular, we prove Theorem \ref{thm:1-6} in the Introduction.

Let $\mathrm{GL}_d(\mathbb Z)$ be the subgroup of the general linear group $\mathrm{GL}_d(\mathbb R)$ consisting of all elements $g \in \mathrm{GL}_d(\mathbb R)$ with $g(\mathbb Z^d)=\mathbb Z^d$.
If we identify each element of $\mathrm{GL}_d(\mathbb R)$ with $d \times d$ non-singular matrix in a standard way, then $\mathrm{GL}_d(\mathbb Z)$ may be considered as the set of unimodular matrices.
For a rational $d$-polytope $P \subset \mathbb R^d$, we define
$$\mathrm{Aut}_{\mathbb Z}(P)=\{g \in \mathrm{GL}_d (\mathbb Z) \mid g(P)=P+\intvec \mbox{ for some }\intvec \in \mathbb Z^d\}$$
and
$$\mathrm{Aut}_{\mathbb Z}(\Gamma_P)=\{g \in \mathrm{GL}_d (\mathbb Z) \mid \TL_{P,g(\bm v)}=\TL_{P,\bm v} \mbox{ for all }\bm v \in \mathbb R^d\}.$$

\begin{proposition}
\label{autom}
For a rational $d$-polytope $P \subset \mathbb R^d$,
one has $\mathrm{Aut}_{\mathbb Z}(\Gamma_P)=\mathrm{Aut}_{\mathbb Z}(P)$.
\end{proposition}

\begin{proof}
We first prove ``$\subset$".
Let $g \in \mathrm{Aut}_{\mathbb Z}(\Gamma_P)$.
Then, for any $\bm v \in \mathbb R^d$, we have
$$\TL_{P,\bm v}(t)=\TL_{P,g(\bm v)}(t)=\# \big( \big(tP + g(\bm v)\big) \cap \mathbb Z^d \big) = \# \big( (tg^{-1}(P)+\bm v) \cap \mathbb Z^d \big)=\TL_{g^{-1}(P),\bm v}(t)$$
for all $t \in \mathbb Z_{\geq 0}$.
Thus we have $\Gamma_P=\Gamma_{g^{-1}(P)}$ so $P=g^{-1}(P)+ \intvec $ for some $\intvec \in \mathbb Z^d$ by Theorem \ref{thm:6-1}.
Then $g \in \mathrm{Aut}_{\mathbb Z}(P)$ since
$P=g(P)-g(\intvec)$ and $g(\intvec) \in \mathbb Z^d$.

We next prove ``$\supset$".
Let $g \in \mathrm{Aut}_{\mathbb Z}(P)$.
Then for any $\bm v \in \mathbb R^d$, we have
$$\# \big( \big( tP+g(\bm v) \big)\cap \mathbb Z^d \big) =\#\big( \big(t g(P) + g(\bm v)\big) \cap  \mathbb Z^d \big) =\#\big( (tP+\bm v)\cap \mathbb Z^d\big)$$
for any $t \in \mathbb Z_{\geq 0}$,
where the last equality follows from the fact that $g \in \mathrm{GL}_d(\mathbb Z)$.
This implies $\TL_{P,g(\bm v)}=\TL_{P,\bm v}$ for all $\bm v \in \mathbb R^d$.
\end{proof}

\begin{example}
Consider the rhombus $Q$ in Example \ref{ex:4.2}.
From the list of translated lattice points enumerators in the example,
one can see that they are equal on $E_1,E_2,E_7$ and $E_8$.
This can be explained using the symmetry.
Let $\rho_1,\rho_2 \in \mathrm{GL}_2(\mathbb Z)$ be a reflection by the $x$-axis and the $y$-axis, respectively.
Then $\rho_1,\rho_2$ do not change $Q$ so they are elements of $\mathrm{Aut}_{\mathbb Z}(Q)$.
We have
\[
\rho_1(E_1)=E_7,\
\rho_1(E_2)=E_8,
\mbox{ and }
\rho_2(E_1)=E_2,
\]
which say that translated lattice points enumerators are equal on $E_1,E_2,E_7$ and $E_8$.
\end{example}

We now focus on centrally symmetric polytopes.
Recall that a quasi-polynomial $f$ is said to be symmetric if its $k$th constituent equals its $(-k)$th constituent for all $k \in \mathbb Z$.

We first prove the following criterion for the symmetry of Ehrhart quasi-polynomials of $P+\bm v$.

\begin{lemma}
\label{symcriteiron}
Let $P \subset \mathbb R^d$ be a rational $d$-polytope. The following conditions are equivalent.
\begin{itemize}
    \item[(i)] $\ehr_{P+\bm v}$ is symmetric for all $\bm v \in \mathbb Q^d$.
    \item[(ii)] For all $\bm v \in \mathbb Q^d$ and $k \in \mathbb Z_{\geq 0}$, one has
    $$\mbox{the $k$th constituent of $\TL_{P,\bm v}$}=\mbox{the $(-k)$th constituent of $\TL_{P,-\bm v}$}.$$
    \end{itemize}
\end{lemma}

\begin{proof}
We first prove ``(i) $\Rightarrow$ (ii)".
Fix $\bm v \in \mathbb Q^d$ and $k \in \mathbb Z_{\geq 0}$.
Then \begin{align*}&\mbox{$k$th constituent of $\TL_{P,\bm v}$}\\
&=\mbox{$k$th constituent of $\ehr_{P+\frac 1 k \bm v}$} &\mbox{(by Lemma \ref{ehrToTrans})}\\
&=\mbox{$(-k)$th constituent of $\ehr_{P+\frac 1 k \bm v}$}  &\mbox{(by (i))}\\
&=\mbox{$(-k)$th constituent of $\TL_{P, -\bm v}$}  &\mbox{(by Lemma \ref{ehrToTrans})},
\end{align*}
as desired.

The proof for ``(ii) $\Rightarrow$ (i)" is similar.
Indeed, we have
\begin{align*}&\mbox{$k$th constituent of $\ehr_{P+\bm v}$}\\
&=\mbox{$k$th constituent of $\TL_{P,k\bm v}$} &\mbox{(by Lemma \ref{ehrToTrans})}\\
&=\mbox{$(-k)$th constituent of $\TL_{P,-k\bm v}$}  &\mbox{(by (ii))}\\
&=\mbox{$(-k)$th constituent of $\ehr_{P+\bm v}$} &\mbox{(by Lemma \ref{ehrToTrans})},
\end{align*}
as desired.
\end{proof}

Recall that a polytope $P \subset \mathbb R^d$ is said to be centrally symmetric if $-P=P+\bm x$ for some $\bm x \in \mathbb R^d$.

\begin{corollary}
\label{projsym}
Let $P \subset \mathbb R^d$ be a rational $d$-polytope.
If $\ehr_{P+ \bm v}$ is symmetric for all $\bm v \in \mathbb Q^{d}$, then
\begin{itemize}
    \item[(i)] $P$ is centrally symmetric.
    \item[(ii)] $\ehr_{\pi_i(P)+\bm u}$ is symmetric for all $\bm u \in \mathbb Q^{d-1}$ and $i \in \{1,2,\dots,d\}$.
\end{itemize}
\end{corollary}

\begin{proof}
(i)
Let $q$ be a positive integer such that $qP$ is integral.
It suffices to prove that $qP$ is centrally symmetric.
Since $\ehr_{qP+q\bm v}(\ell )=\ehr_{P+\bm v}(q\ell )$ for all $\ell \in \mathbb Z_{\geq 0}$,
the $k$th constituent of $\ehr_{qP+q\bm v}$ is obtained from the $qk$th constituent of $\ehr_{P+\bm v}$ by substituting $t$ with $\frac t q$ (as polynomials in 
$t$).
This fact and the assumption say that $\ehr_{qP+q\bm v}$ is symmetric for all $\bm v \in \mathbb Q^d$.
Since $qP$ is a lattice polytope,
it follows from Theorem \ref{thm:deVriesYoshinaga} that $qP$ is centrally symmetric.

(ii)
For any $\bm v \in \mathbb Q^d$, we have
\begin{align*}
    &\mbox{$k$th constituent of $\TL_{\pi_i(P),\pi_i(\bm v)}$}\\
    &=\mbox{$k$th constituent of $\sum_{0 \leq s <1} \TL_{P,\bm v+s\mathbf e_i}^{(-i)}$} & (\mbox{by Proposition \ref{7.6}})\\
    &=\mbox{$k$th constituent of $\sum_{0 \leq s <1} \big(\TL_{P,\bm v+s\mathbf e_i}-\TL_{P,\bm v+s\mathbf e_i-\varepsilon_s \mathbf e_i} \big)$}\\
    &=\mbox{$(-k)$th constituent of $\sum_{0 \leq s <1} \big(\TL_{P,-\bm v-s\mathbf e_i}-\TL_{P,-\bm v -s\mathbf e_i +\varepsilon_s \mathbf e_i} \big)$} & (\mbox{by Lemma \ref{symcriteiron}})\\
    &=\mbox{$(-k)$th constituent of $\sum_{0 \leq s <1} \TL_{P,-\bm v-s\mathbf e_i}^{(+i)}$} \\
    &=\mbox{$(-k)$th constituent of $\TL_{\pi_i(P),-\pi_i(\bm v)}$} & (\mbox{by Proposition \ref{7.6*}}),
\end{align*}
where each $\varepsilon_s$ is a sufficiently small positive number which depends on $s$. 
This proves that $\pi_i(P)$ satisfies the condition (ii) of Lemma \ref{symcriteiron}.
\end{proof}

We now come to the goal of this section.
Let $P$ be a centrally symmetric polytope with $-P=P+\bm x$.
Then $\frac 1 2 \bm x$ is a center of $P$, and if $\bm p =\frac 1 2 \bm x +\bm p'$ is a vertex of $P$,
the point $\frac 1 2 \bm x -\bm p'$ is also a vertex of $P$ by the central symmetry.
We write this vertex $\frac 1 2 \bm x -\bm p'$ as $\bm p^*$.

\begin{theorem}
\label{thm:7.5}
Let $P \subset \mathbb R^d$ be a rational $d$-polytope.
The following conditions are equivalent.
\begin{itemize}
    \item[(i)] $\ehr_{P+\bm v}$ is symmetric for all $\bm v \in \mathbb Q^d$.
    \item[(ii)] $P$ is centrally symmetric and $\bm p -\bm p^* \in \mathbb Z^d$ for every vertex $\bm p $ of $P$.
    \end{itemize}
\end{theorem}

To prove the theorem,
we recall the following  basic fact on $\mathbb Z$-modules.

\begin{lemma}
\label{zbasiscor}
Let $X \subset \mathbb R^d$ be a $d$-dimensional cone with apex $\bm 0$.
There is a $\mathbb Z$-basis of $\mathbb Z^d$ which is contained in $\mathrm{int}(X)$.
\end{lemma}

\begin{proof}
Take any integer vector $\bm n \in \mathrm{int}(X)$
with $\gcd(\bm n)=1$. 

First, we claim that 
there are $\bm n_1,\dots,\bm n_{d-1} \in \mathbb Z^d$ such that
$\bm n,\bm n_1,\dots,\bm n_{d-1}$ is a $\mathbb Z$-basis of $\mathbb Z^d$.
In fact, by the assumption, the $\mathbb Z$-module $\mathbb Z^d/( \mathbb Z\bm n )$ is a free $\mathbb Z$-module of rank $d-1$ since it is torsionfree. 
If we choose $\bm n_1,\dots,\bm n_{d-1} \in \mathbb Z^d$ so that they form a $\mathbb Z$-basis for $\mathbb Z^d/(\mathbb Z\bm n)$,
the sequence $\bm n,\bm n_1,\dots,\bm n_{d-1}$ becomes a $\mathbb Z$-basis of $\mathbb Z^d$.

Now, we show that we can choose $\bm n_1,\ldots,\bm n_{d-1}$ from $\mathrm{int}(X)$. 
For each $\bm n_i$, since $\bm n $ is in the interior of $X$, by taking a sufficiently large integer $k_i$, the point $\bm n_i + k_i \bm n$ is contained in $\mathrm{int}(X)$.
Then $\bm n, \bm n_1+k_1\bm n,\dots,\bm n_{d-1}+ k_{d-1} \bm n$ is a desired $\mathbb Z$-basis.
\end{proof}

\begin{proof}[Proof of Theorem \ref{thm:7.5}]
((ii) $\Rightarrow$ (i))
By taking an appropriate translation, we may assume $P=-P$.
Then $\bm p^*=-\bm p$ for every vertex $\bm p$ of $P$,
so the condition (ii) says that $2P$ is integral.
In particular, every quasi-polynomial $\TL_{P,\bm v}$ has period 2.
We prove that $P$ satisfies the condition (ii) of Lemma \ref{symcriteiron}.

Let $\bm v \in \mathbb Q^d$ and $k \in \{0,1\}$.
Since $P=-P$, Lemma \ref{lem:cs} says
\begin{align*}
\mbox{the $k$th constituent of $\TL_{P,\bm v}$}
&=\mbox{the $k$th constituent of $\TL_{P,-\bm v}$}.
\end{align*}
However, since $\TL_{P,-\bm v}$ has period $2$,
the RHS in the above equation equals
the $(-k)$th constituent of $\TL_{P,-\bm v}$.

((i) $\Rightarrow$ (ii))
We have already seen that (i) implies that $P$ is centrally symmetric in Corollary \ref{projsym}.
We prove the second condition of (ii) by induction on $d$.
Suppose $d=1$. Then we may assume
$$\textstyle P=[0,x+ \frac p q]$$
for some $x,p,q\in \mathbb Z_{\geq 0}$ with $0 \leq p <q$.
Then we have
$$
\textstyle
\mbox{ the first constituent of $\ehr_P$}= \mathrm{vol}(P)t- \frac p q + 1
$$
and
$$
\textstyle
\mbox{ the $(q-1)$th constituent of $\ehr_P$}= \mathrm{vol}(P)t-(q-1) \frac p q +\lfloor (q-1) \frac p q \rfloor+1.
$$
Then the condition (i) says $p-2p/q=\lfloor p(q-1)/q \rfloor$, but it implies $2p/q \in \mathbb Z$.
Hence $2P$ is integral which guarantees the condition (ii).

Suppose $d>1$.
Let $\bm p \in \mathbb Q^d$ be a vertex of $P$.
Consider the normal cone at the vertex $\bm p$
$$X=\{ \bm a \in \mathbb R^d\mid \max\{(\bm a,\bm x)\mid \bm x \in P\}=(\bm a,\bm p)\}.$$
This is a $d$-dimensional cone with apex $\bm 0$.
By Lemma \ref{zbasiscor},
there is a $\mathbb Z^d$-basis $\mathbf e'_1,\dots,\mathbf e'_d$ which is contained in $ \mathrm{int}(X)$.
Consider the linear transformation $g \in \mathrm{GL}_n(\mathbb Z)$ which changes the hyperplane $\{ \bm x \in \mathbb R^d \mid (\bm x,\mathbf e'_i)=0\}$
to $\{ \bm x=(x_1,\dots,x_d) \in \mathbb R^d \mid x_i=0\}$.
Since $g(\mathbb Z^d)=\mathbb Z^d$,
we have $\ehr_{g(P)+\bm u}(t)=\ehr_{P+g^{-1}(\bm u)}(t)$ for all $\bm u \in \mathbb Q^d$,
so $g(P)$ also satisfies the condition (i).
Let $g(\bm p)=(y_1,\dots,y_d)$.
By the choice of $\mathbf e_1',\dots,\mathbf e_d'$, we have 
$$g(P) \cap \{(x_1,\dots,x_d)\in \mathbb R^d \mid x_i=y_i\}=\{g (\bm p)\} \ \ \mbox{ for all $1 \leq i \leq d$}. $$
This says that $\pi_j(g(\bm p))$ is a vertex of $\pi_j(g(P))$ for $j=1,2,\dots,d$ and the same holds for $\pi_j(g(\bm p^*))$ by the central symmetry.
For each $j=1,2,\dots,d,$
Lemma \ref{projsym} says that
$\pi_j(g(P))$ satisfies the condition (i),
so we have that $\pi_j(g(\bm p)-g(\bm p^*)) \in \mathbb Z^{d-1}$ by the induction hypothesis.
But then we must have $g(\bm p- \bm p^*) \in \mathbb Z^d$ and therefore $\bm p-\bm p^* \in \mathbb Z^d$.
\end{proof}

If $P$ is a centrally symmetric polytope with the center $\bm c$ and $\bm p$ is a vertex of $P$,
then $\bm p-\bm p^*=2(\bm p-\bm c)$,
so Theorem \ref{thm:7.5} is equivalent to Theorem \ref{thm:1-6} in the Introduction.

\section{A connection to commutative algebra}
In this section,
we briefly explain a connection between translated lattice points enumerators and conic divisorial ideals of Ehrhart rings in commutative algebra.
In particular, we explain that Theorem \ref{thm:McMullen} can be proved algebraically using the duality of Cohen--Macaulay modules.

\subsection{Conic divisorial ideals}
Let $S=\mathbb F[x_1^\pm,\dots,x_{d+1}^\pm]$ be the Laurent polynomial ring over a field $\mathbb F.$
We will consider the grading of $S$ defined by $\deg(x_1)=\cdots=\deg(x_d)=0$ and $\deg(x_{d+1})=1$.
For $\bm a=(a_1,\dots,a_{d+1}) \in \mathbb Z^{d+1}$,
we write 
$$x^{\bm a}=x_1^{a_1} \cdots x_{d+1}^{a_{d+1}}.$$
Let $P \subset \mathbb R^d$ be a rational $d$-polytope.
The \textbf{Ehrhart ring} $\mathbb F[P]$ of $P$ (over $\mathbb F$) is the monoid algebra generated by the monomials $x^{\bm a}$ such that $\bm a$ is in the monoid $\mathcal C_P \cap \mathbb Z^{d+1}$.
As vector spaces, we can write
\begin{align}
    \label{3-1}
    \mathbb F[P]=\mathrm{span}_{\mathbb F}\{ x^{\bm a} \mid \bm a \in \mathcal C_P \cap \mathbb Z^{d+1}\}.
\end{align}
For a finitely generated graded $\mathbb F[P]$-module $M$, its {\bf Hilbert function} is the function defined by $\mathrm{hilb}(M,k)=\dim_{\mathbb F} M_k$ for $k \in \mathbb Z$,
where $M_k$ is the degree $k$ component of $M$,
and the {\bf Hilbert series} of $M$ is the formal power series $\Hilb(M,z)=\sum_{k \in \mathbb Z} \mathrm{hilb}(M,k) z^k$.
Ehrhart rings are closely related to Ehrhart quasi-polynomials.
Indeed, from \eqref{3-1},
we can see that the Hilbert function of $\mathbb F[P]$ is nothing but the Ehrhart quasi-polynomial of $P$.

For any $\bm v\in \mathbb R^{d+1}$,
the vector space
$$I_{\bm v}=\mathrm{span}_{\mathbb F} \{ x^{\bm a} \mid \bm a \in (\mathcal C_P + \bm v)\cap \mathbb Z^{d+1}\} \subset S$$
becomes a finitely generated graded $\mathbb F[P]$-module.
The modules $I_{\bm v}$ are called \textbf{conic divisorial ideals} of $\mathbb F[P]$.
We note that different vectors in $\mathbb R^{d+1}$ could give the same conic divisorial ideal,
more precisely, we have $I_{\bm v}=I_{\bm u}$ if and only if the cones $\mathcal C_P+\bm v$ and $\mathcal C_P + \bm u$ have the same lattice points.

Let us call a conic divisorial ideal $I$ \textbf{standard} if $I=I_{(\bm v,0)}$ for some $\bm v \in \mathbb R^d$.
Hilbert functions of standard conic divisorial ideals are nothing but translated lattice points enumerators.
Indeed, for any $\bm v \in \mathbb R^d$, we have
{\small
\begin{align}
    \label{3-2}
    \dim_{\mathbb F}(I_{(\bm v,0)})_t =\# \{\bm a\!=\!(a_1,\dots,a_{d+1}) \in (\scC_P+(\bm v,0)) \cap \mathbb Z^{d+1}\! \mid a_{d+1}\!=\!t\}=\# \big((tP+\bm v)\! \cap \mathbb Z^d \big).
\end{align}
}

We will not explain algebraic backgrounds on (conic) divisorial ideals of Ehrhart rings since it is not relevant to the theme of this paper.
But in the rest of this section we briefly explain how algebraic properties of conic divisorial ideals can be used to consider properties of translated lattice points enumerators.
For more detailed information on conic divisorial ideals,
see \cite{BG} and \cite[\S 4.7]{BG:book}.

\subsection{Hilbert series of conic divisorial ideals
and an algebraic proof of Theorem \ref{thm:McMullen}}

We need some basic tools on commutative algebra such as the Cohen--Macaulay property and canonical modules.
We refer the readers to \cite[\S3 and \S4]{BH} for basics on commutative algebra.

We introduce one more notation.
For $\bm v \in \mathbb R^{d+1}$,
we define
\begin{align}
    \label{3-3}
    I^\circ_{\bm v}=\mathrm{span}_{\mathbb F} \big\{x^{\bm a} \mid \bm a \in \big(\mathrm{int}(\mathcal C_P)+\bm v\big) \cap \mathbb Z^{d+1}\big\}.
\end{align}
The space $I^\circ_{\bm v}$ is also a conic divisorial ideal. Indeed, if $\bm w \in \mathrm{int}(\mathcal C_P)$ is a vector which is sufficiently close to the origin, then we have
$$ \big(\mathrm{int}(\mathcal C_P)+ \bm v \big)\cap \mathbb Z^{d+1}=(\mathcal C_P+\bm v + \bm w)\cap \mathbb Z^{d+1},$$
which says $I_{\bm v}^\circ = I_{\bm v+\bm w}$.
The following facts are known. See \cite[Corollary 3.3 and Remark 4.4(b)]{BG:book}.
\begin{itemize}
    \item $I_{\bm v}$ is a $(d+1)$-dimensional Cohen--Macaulay module.
    \item $I_{\bm v}^\circ$ is the canonical module of $I_{-\bm v}$, more precisely, we have
$$\Hom_{\mathbb F[P]}(I_{\bm v},\omega) \cong \mathrm{span}_{\mathbb F}\{x^{\bm a} \mid \bm a \in (\mathrm{int}(\mathcal C_P)-\bm v) \cap \mathbb Z^{d+1}\}=I_{-\bm v}^\circ,$$
    where
    $\omega=\mathrm{span}_{\mathbb F}\{x^{\bm a}\mid \bm a \in \mathrm{int}(\mathcal C_P) \cap \mathbb Z^{d+1}\}$ is the graded canonical module of $\mathbb F[P]$.
\end{itemize}
These properties give the following consequences on Hilbert series of conic divisorial ideals.

\begin{proposition}
    \label{prop:hilbseries}
    Let $P \subset \mathbb R^d$ be a rational $d$-polytope and $q$ the denominator of $P$.
    Let $\bm v=(v_1,\dots,v_{d+1}) \in \mathbb R^{d+1}$ and $\alpha=\lceil v_{d+1} \rceil$.
    \begin{itemize}
        \item[(1)] $\Hilb(I_{\bm v}^\circ,z)=(-1)^{d+1} \Hilb(I_{-\bm v},z^{-1})$.
        \item[(2)] $\Hilb(I_{\bm v},z)=\frac {z^\alpha} { (1-z^q)^{d+1} } Q(z)$ for some polynomial $Q(z) \in \mathbb Z_{\geq 0}[z]$ of degree $<q(d+1)$.
    \end{itemize}
\end{proposition}

\begin{proof}
The equality (1) is the well-known formula of the Hilbert series of a canonical module. See \cite[Theorem 4.45]{BH}.
We prove (2).
Consider the subring
\[
A=\mathrm{span}_{\mathbb F} \big\{ x^{\bm a} \mid x^{\bm a} \in \mathcal C_P \mbox{ and  $\deg (x^{\bm a}) \in q \mathbb Z$} \big \} \subset \mathbb F[P].
\]
Since $q P$ is integral, $\mathbb F[qP]$ is a semi-standard graded $\mathbb F$-algebra, that is, $\mathbb F[qP]$ is a finitely generated as a module over a standard graded $\mathbb F$-algebra $\mathbb F[x^{\bm a}x_{d+1}\mid x^{\bm a} \in P \cap \mathbb Z^d]$
(see \cite[Theorem 9.3.6]{Villarreal}(d)).
Then, since $A \cong \mathbb F[qP]$, where the degree $k$ part of $\mathbb F[qP]$ corresponds to the degree $qk$ part of $A$,
any finitely generated $A$-module $M$ of Krull dimension $m$ has the Hilbert series of the form $Q(z)/(1-z^q)^{m}$ for some polynomial $Q(z)$,
and if $M$ is Cohen--Macaulay then $Q(z) \in \mathbb Z_{\geq 0}[z]$
(\cite[Corollaries 4.8 and 4.10]{BH}).

Since $\mathbb F[P]$ is a finitely generated $A$-module,
$I_{\bm v}$ is a finitely generated Cohen--Macaulay $A$-module of Krull dimension $d+1$. Thus there is a polynomial $Q(z) \in \mathbb Z_{\geq 0}[z]$ such that
$$\Hilb(I_{\bm v},z)=\frac 1 {(1-z^q)^{d+1}} Q(z).$$
Since $(I_{\bm v})_k=0$ for $k <\alpha = \lceil v_{d+1} \rceil$ by the definition of $I_{\bm v}$,
the polynomial $Q(z)$ must be of the form
$$Q(z)=c_0t^\alpha+c_1t^{\alpha+1}+ \cdots +c_m t^{\alpha+m}$$
for some $m \geq 0$,
where $c_0,\dots,c_m \in \mathbb Z_{\geq 0}$ and $c_m \ne 0$, so it follows that
$$\Hilb(I_{\bm v},z)=\frac {z^{\alpha}} {(1-z^q)^{d+1}} (c_0+c_1z+ \cdots +c_mz^m). $$
Now it remains to prove $m <q(d+1)$.
By statement (1), we have
\begin{align*}
    \Hilb(I_{-\bm v}^\circ,z)&=(-1)^{d+1} \frac {z^{-\alpha}} {(1-z^{-q})^{d+1}} (c_0+c_1z^{-1}+\cdots+c_mz^{-m})\\
    &=\frac {z^{q (d+1)-\alpha-m}} {(1-z^q)^{d+1}} (c_m + c_mz + \cdots + c_0 z^m).
\end{align*}
This says
$$-\alpha < \min\{k \in \mathbb Z \mid (I_{-\bm v}^\circ)_k \ne 0\} = q(d+1)-\alpha -m$$
proving the desired inequality $m<q(d+1)$.
\end{proof}

The statements in Proposition \ref{prop:hilbseries} are known to imply the quasi-polynomiality and reciprocity of translated lattice points enumerators in Theorem \ref{thm:McMullen}.
Recall that $\TL_{P,\bm v}$ coincides with the Hilbert function of $I_{(\bm v,0)}$.
Proposition \ref{prop:hilbseries}(2) says that
the Hilbert series of $I_{(\bm v,0)}$ can be written in the form $\frac 1 {(1-z^q)^{d+1}}Q(z)$ for some polynomial $Q(z)$ of degree $<q(d+1)$,
which is known to imply that $\mathrm{hilb}(I_{(\bm v,0)},t)$
$(=\TL_{P,\bm v}(t)$) coincides with a quasi-polynomial with period $q$ for $t \geq 0$.
See e.g., \cite[\S 3.8]{BR:2007} or \cite[\S 4]{Stanley}.
Also, Proposition \ref{prop:hilbseries}(1) is essentially equivalent to the reciprocity in Theorem \ref{thm:McMullen}(2).
See \cite[\S 4.3]{BR:2007}.

Finally,
we note that the proposition gives some restriction to the possible values of $\TL_{P,\bm v}$.
If $P \subset \mathbb R^d$ is a lattice $d$-polytope and $\bm v \in \mathbb R^d$,
then the proposition says
$$\Hilb(I_{(\bm v,0)},z)=\frac 1 {(1-z)^{d+1}}(h_0+h_1z+ \cdots +h_d z^d)$$
for some $h_0,h_1,\dots,h_d \in \mathbb Z_{\geq 0}$.
These $h$-numbers must satisfy the following conditions
\begin{itemize}
    \item[(I)] $h_0=1$ if $\bm v \in \mathbb Z^d$ and $h_0=0$ if $\bm v \not \in \mathbb Z^d$;
    \item[(II)] $h_0+\cdots +h_d=d! \mathrm{vol}(P)$.
\end{itemize}
The first condition follows from $h_0=\dim_{\mathbb F}(I_{(\bm v,0)})_0$,
and the second condition follows since $\frac 1 {d!} (h_0+\cdots +h_d)$ is the top degree coefficient of the polynomial $\mathrm{hilb}(I_{(\bm v,0)},t)$.
Below we give a simple application of this.
Consider a lattice polygon  $P \subset \mathbb R^2$ whose volume is $\frac 3 2$.
Then the possible values of $h_0+h_1z+h_2z^2$ are
\[
1+z+z^2,\ 
1+2z,\ 
1+2z^2,\ 
3z,\ 
2z+z^2,\ 
z+2z^2,\ 
3z^2.
\]
If $f(t)$ is a polynomial $\sum_{t=0}^\infty f(t)z^t=\frac 1 {(1-z)^3} (h_0+h_1z+h_2z^2)$,
then $f(t)=h_0{t+2 \choose 2}+h_1 {t+1 \choose 2}+h_2 {t \choose 2}$.
So a translated lattice points enumerator of an integral polygon with volume $\frac 3 2$ must be one of the following polynomials
\[
\textstyle
\frac 3 2 t^2 + \frac 3 2 t +1,\
\frac 3 2 t^2 + \frac 5 2 t +1,\
\frac 3 2 t^2 + \frac 1 2 t +1,\
\frac 3 2 t^2 + \frac 3 2 t,\
\frac 3 2 t^2 + \frac 1 2 t,\
\frac 3 2 t^2 - \frac 1 2 t,\
\frac 3 2 t^2 - \frac 3 2 t.
\]
Four of them appear as translated lattice points enumerators of the trapezoid in the Introduction.
See \eqref{eq1-1}.

\begin{remark}
Alhajjar \cite[\S 4]{alh} studied the numbers $h_0,h_1,\dots,h_d$ mentioned above by a more combinatorial approach
and proved various results including (I) and (II).
\end{remark}

\section{Problems}

In this last section,
we list a few problems which we cannot answer.

\subsection*{Gcd property and zonotopes}
A quasi-polynomial $f$ with period $q$ is said to have the {\bf gcd property} if its $k$th constituent only depends on the gcd of $k$ and $q$ for all 
$k \in \mathbb Z$.
We note that if $f$ has the gcd property then $f$ must be symmetric.
It was proved in \cite{DY:2021} that,
for a lattice $d$-polytope $P \subset \mathbb R^d$,
$\ehr_{P+\bm v}$ has the gcd property for all $\bm v \in \mathbb Q^d$ if and only if $P$ is a zonotope.
Considering the statement in Theorem \ref{thm:1-6},
one may ask if a similar statement holds for zonotopes $P$ such that $2P$ is integral,
but this is not the case.
Indeed, the rhombus $Q$ in Example \ref{ex:4.2}
is a zonotope and $2Q$ is integral but the computation given in the example says that
$\ehr_{Q+(\frac 1 8, \frac 1 8)}$ does not satisfy the gcd property.
We repeat the following question asked in \cite[Problem 6.7(2)]{DY:2021}.

\begin{problem}
Let $P \subset \mathbb R^d$ be a rational $d$-polytope.
Is it true that,
if $\ehr_{P+ \bm v}$ has the gcd property for all $\bm v \in \mathbb Q^d$,
then $P=Q+\bm u$ for some integral zonotope $Q$ and some $\bm u \in \mathbb Q^d$?
\end{problem}

To consider this problem
we can assume that $P$ is a zonotope by the argument similar to the proof of Corollary \ref{projsym}(i)
and $2P$ is integral by Theorem \ref{thm:1-6}.

\subsection*{Period collapse}
Recall that the denominator of a rational polytope $P$ is always a period of $\ehr_{P}$.
If the minimum period of $\ehr_{P}$ is not equal to the denominator of $P$, we say that period collapse occurs to $P$.
A period collapse is a major subject in the study of Ehrhart quasi-polynomials (see e.g., \cite{BSW,HM,MM,MW}).
We ask the following vague question:
Can a relation between $\ehr_{P+\bm v}$ and $\TL_{P,\bm v}$ be used to produce polytopes giving period collapse?
For translations of a lattice polytope,
a period collapse cannot occur.
Indeed, if $P$ is a lattice polytope, then the minimum period of $\ehr_{P+\bm v}$ must be the smallest integer $k$ such that $k \bm v$ is integral since 
the constant term of the $k$th constituent of $\ehr_{P+\bm v}$ is $\TL_{P,k\bm v}(0)=\# (\{k\bm v\} \cap \mathbb Z^n)$,
which is non-zero only when $k\bm v$ is integral.


\begin{thebibliography}{9}


\bibitem{Alexandrov}
A.D. Alexandrov, Convex Polyhedra, Springer-Verlag, Berlin, 2005. 

\bibitem{alh}
E. Alhajjar, 
A New Valuation On Lattice Polytopes. 
Ph.D. Thesis (2017).

\bibitem{ABM}
F. Ardila, M. Beck, J. McWhirter,
The Arithmetic of Coxeter Permutahedra.
\textit{Rev. Acad.
Colombiana Cienc. Exact. F\'is. Natur.} \textbf{44} (2020), 1152--1166.


\bibitem{BBKV}
V. Baldoni, N. Berline, M. K\"oppe, and M. Vergne,
Intermediate sums on polyhedra: computation and real Ehrhart theory, \textit{Mathematika} \textbf{59} (2013), 1--22.

\bibitem{BER}
M. Beck, S. Elia, and S. Rehberg,
Rational Ehrhart Theory,
\textit{Integers} \textbf{23} (2023), Paper No.\ A60,
31 pp.

\bibitem{BR:2007} M. Beck and S. Robins, ``Computing the Continuous Discretely'',  Undergraduate Texts in Mathematics, Springer, 2007.

\bibitem{BSW}
M. Beck, S.V. Sam, K.M. Woods,
Maximal periods of (Ehrhart) quasi-polynomials,
\textit{J. Combin. Theory Ser. A} \textbf{115} (2008), 517--525.

\bibitem{Bruns}
W. Bruns, 
Conic divisor classes over a normal monoid algebra,  in: Commutative Algebra and Algebraic Geometry, in: Contemp. Math., vol. 390, Amer. Math. Soc., 2005, pp. 63--71.

\bibitem{BH}
W. Bruns and J. Herzog,
Cohen--Macaulay rings,
Cambridge Studies in Advanced Mathematics, vol. 39, Cambridge University Press, Cambridge, 1993. 

\bibitem{BG}
W. Bruns and J. Gubeladze,
Divisorial linear algebra of normal semigroup rings, \textit{Algebr. Represent. Theory} \textbf{6} (2003),
139–-168.

\bibitem{BG:book}
W. Bruns, J. Gubeladze, Polytopes, Rings and K-Theory, Springer Monographs in Mathematics, Springer, Dordrecht, 2009. 

\bibitem{DY:2021} 
C. de\ Vries and M. Yoshinaga, 
Ehrhart Quasi-Polynomials of Almost Integral Polytopes. to appear in 
Discrete Comput Geom. 

\bibitem{Gru:book}
B. Gr\"unbaum,
Convex polytopes,
Second edition,
Graduate Texts in Math. \textbf{221},
Springer-Verlag, New York, 2003.

\bibitem{HM}
C. Haase and T.B. McAllister,
Quasi-period collapse and $GL_n(\mathbb Z)$-scissors congruence in rational polytopes, in: Contemp. Math.,
vol. 452, 2008, pp. 115--122.


\bibitem{Linke}
E. Linke,
Rational Ehrhart quasi-polynomials, \textit{J. Combin. Theory Ser. A} \textbf{118} (2011), no. 7, 1966--1978.




\bibitem{Royer1}
T. Royer,
Reconstruction of rational polytopes from the real-parameter Ehrhart function of its translates,
arXiv:1712.01973.


\bibitem{Royer2}
T. Royer,
Reconstruction of symmetric convex bodies from Ehrhart-like data,
arXiv:1712.03937.

\bibitem{MM}
T.B. McAllister and M. Moriarity,
Ehrhart quasi-period collapse in rational polygons,
\textit{Journal of Combinatorial Theory. Series A} \textbf{150} (2017), 377--385.

\bibitem{MW}
T.B. McAllister and K.M. Woods,
The minimum period of the Ehrhart quasi-polynomial of a rational polytope,
\textit{J. Combin. Theory
Ser. A} \textbf{109} (2005), 345--352.

\bibitem{McMullen}
P. McMullen,
Lattice invariant valuations on rational polytopes.
Arch. Math. (Basel) 31 (1978/79), no. 5, 509–516. 




\bibitem{Stanley}
R.P. Stanley,
Enumerative combinatorics. Volume 1. Second edition. Cambridge Studies in Advanced Mathematics, 49. Cambridge University Press, Cambridge, 2012. 

\bibitem{Villarreal}
R.H. Villarreal,
Monomial Algebras, second edition, Monographs and Research Notes in Mathematics, CRC Press, Boca Raton, FL, 2015.

\bibitem{Z}
G. Ziegler,
Lectures on polytopes,
Fraduate Texts in Mathematics, vol. {152}, Springer-Verlag, New York, 1995. 

\end{thebibliography}
\end{document}